\pdfoutput=1
\RequirePackage{ifpdf}
\ifpdf 
\documentclass[pdftex]{sigma}
\else
\documentclass{sigma}
\fi

\usepackage{bm}

\def\N{\mathbb{N}}
\def\W{\mathbb{W}}
\def\C{\mathbb{C}}
\def\R{\mathbb{R}}

\def\sH{{\sf H}}
\def\sU{{\sf U}}

\def\bra{\langle}
\def\ket{\rangle}

\def\x{\bm{x}}

\def\1{{\bf 1}}

\def\cM{{\cal M}}
\def\cS{{\cal S}}
\def\cP{{\cal P}}
\def\cB{{\cal B}}
\def\cC{{\cal C}}
\def\cR{{\cal R}}

\def\rD{{\rm D}}
\def\rBW{{\rm BW}}
\def\rchD{{\rm chD}}

\def\rw{{\rm w}}
\def\rmm{{\rm m}}
\def\rs{{\rm s}}
\def\rd{{\rm d}}
\def\rk{{\rm k}}
\def\rt{{\rm t}}

\def\dist={\stackrel{\rm d}{=}}

\numberwithin{equation}{section}

\newtheorem{Theorem}{Theorem}[section]
\newtheorem*{Theorem*}{Theorem}

\newtheorem{Lemma}[Theorem]{Lemma}
\newtheorem{Proposition}[Theorem]{Proposition}
 { \theoremstyle{definition}

\newtheorem{Example}[Theorem]{Example}
\newtheorem{Remark}[Theorem]{Remark} }

\begin{document}
\allowdisplaybreaks

\newcommand{\arXivNumber}{2106.00442}

\renewcommand{\thefootnote}{}

\renewcommand{\PaperNumber}{049}

\FirstPageHeading

\ShortArticleName{Functional Equations Solving Initial-Value Problems of Complex Burgers-Type Equations}

\ArticleName{Functional Equations Solving Initial-Value Problems\\ of Complex Burgers-Type Equations\\ for One-Dimensional Log-Gases\footnote{This paper is a~contribution to the Special Issue on Non-Commutative Algebra, Probability and Analysis in Action. The~full collection is available at \href{https://www.emis.de/journals/SIGMA/non-commutative-probability.html}{https://www.emis.de/journals/SIGMA/non-commutative-probability.html}}}

\Author{Taiki ENDO~$^{\rm a}$, Makoto KATORI~$^{\rm a}$ and Noriyoshi SAKUMA~$^{\rm b}$}

\AuthorNameForHeading{T.~Endo, M.~Katori and N.~Sakuma}

\Address{$^{\rm a)}$~Department of Physics, Faculty of Science and Engineering, Chuo University,\\
\hphantom{$^{\rm a)}$}~Kasuga, Bunkyo-ku, Tokyo 112-8551, Japan}
\EmailD{\href{mailto:taiki@phys.chuo-u.ac.jp}{taiki@phys.chuo-u.ac.jp}, \href{mailto:katori@phys.chuo-u.ac.jp}{katori@phys.chuo-u.ac.jp}}

\Address{$^{\rm b)}$~Graduate School of Natural Sciences, Nagoya City University,\\
\hphantom{$^{\rm b)}$}~Mizuho-ku, Nagoya, Aichi 467-8501, Japan}
\EmailD{\href{mailto:sakuma@nsc.nagoya-cu.ac.jp}{sakuma@nsc.nagoya-cu.ac.jp}}

\ArticleDates{Received February 24, 2022, in final form June 23, 2022; Published online July 02, 2022}

\Abstract{We study the hydrodynamic limits of three kinds of one-dimensional stochastic log-gases known as Dyson's Brownian motion model, its chiral version, and the Bru--Wishart process studied in dynamical random matrix theory. We define the measure-valued processes so that their Cauchy transforms solve the complex Burgers-type equations. We show that applications of the method of characteristic curves to these partial differential equations provide the functional equations relating the Cauchy transforms of measures at an arbitrary time with those at the initial time. We transform the functional equations for the Cauchy transforms to those for the $R$-transforms and the $S$-transforms of the measures, which play central roles in free probability theory. The obtained functional equations for the $R$-transforms and the $S$-transforms are simpler than those for the Cauchy transforms and useful for explicit calculations including the computation of free cumulant sequences. Some of the results are argued using the notion of free convolutions.}


\Classification{82C22; 60B20; 44A15; 46L54}

\renewcommand{\thefootnote}{\arabic{footnote}}
\setcounter{footnote}{0}

\section{Introduction and results}

\subsection{Transformations of measures and complex Burgers-type equations}

Among a variety of recent developments in
random matrix theory \cite{ABD11,AGZ10,For10,Meh04},
we study in this paper an intersection of two important topics;
time-dependent random matrix models
\cite{Bru89,Dys62,Kat16,KT04,KO01}
and free probability theory
\cite{BV93,Bia97,Bia97b,MS17,NS06,Voi86,VDN92}.

The cases of the Gaussian unitary ensemble (GUE)
and the chiral GUE are typical eigenvalue distributions
on $\R$ of Hermitian random matrices, and
their dynamical extensions are described by
systems of stochastic differential equations (SDEs)
called
\textit{Dyson's Brownian motion model}~\cite{Dys62}
and its chiral version \cite{Bru91,KT04,KO01}.
(Compare the SDEs (\ref{eqn:Wishart2}) for the chiral version
with (\ref{eqn:Dyson1}) for the original one.
The two terms in the parentheses in the last
term of the r.h.s.\ in (\ref{eqn:Wishart2}) are
\textit{chiral} to each other.)

{\samepage
For $S \subset \R$, let $\cP^0(S)$ be a set of all Borel
probability measures on $S$
with bounded supports
equipped with the weak topology.
For an arbitrary but fixed
$T > 0$, $\cC\big([0, T] \to \cP^0(S)\big)$ denotes the
space of continuous processes defined in the time period $[0, T]$
realized in $\cP^0(S)$.

}

In this manuscript,
we consider the \textit{hydrodynamic limits} of
Dyson's Brownian motion model and
its chiral version with
an additional parameter $\lambda \in \R_+ :=[0, \infty)$
as elements of $\cC\big([0, \infty) \allowbreak\to \cP^0(\R)\big)$,
which are denoted by $(\rw_t)_{t \geq 0}$
and $(\rw_{\lambda, t})_{t \geq 0}$, respectively.
First, we assume that the initial probability measure $\rw_0$
and $\rw_{\lambda, 0}$ are in $\cP^0(\R)$
and the \textit{Cauchy transforms} of measures,
\begin{gather}
G_{\mu}(z) := \int_S \frac{\mu({\rm d}x)}{z-x},
\qquad z \in \C^+ :=\{z \in \C\colon \operatorname{Im} z > 0\},
\label{eqn:Cauchy1}
\end{gather}
are well defined satisfying the condition
$\lim_{y \uparrow \infty}
\sqrt{-1} y G_{\mu}\big(\sqrt{-1} y\big)=1$
for $\mu=\rw_t$ and $\rw_{\lambda, t}$, $t \geq 0$.
It is known that \cite{BN08,BNW13,BNW14,Neu08a,Neu08b,RS93}
given $G_{\rw_0}$ and $G_{\rw_{\lambda, 0}}$
obtained by $\rw_0$ and $\rw_{\lambda,0}$, respectively,
$(G_{\rw_t})_{t \geq 0}$ and $(G_{\rw_{\lambda, t}})_{t \geq 0}$
are uniquely determined by the solutions
of the following partial differential equations (PDEs):
\begin{gather}
\frac{\partial G_{\rw_t}(z)}{\partial t}
+G_{\rw_t}(z) \frac{\partial G_{\rw_t}(z)}{\partial z}=0, \qquad
t \geq 0,\label{eqn:Gwt}
\\
\frac{\partial G_{\rw_{\lambda,t}}(z)}{\partial t}
+\bigg( G_{\rw_{\lambda,t}}(z) - \frac{1-\lambda}{2z} \bigg)
\frac{\partial G_{\rw_{\lambda,t}}(z)}{\partial z}
+\frac{1-\lambda}{2z^2} G_{\rw_{\lambda,t}}(z)=0, \qquad
t \geq 0.\label{eqn:Gm1/2}
\end{gather}
Equation (\ref{eqn:Gwt}) is known as
the \textit{complex Burgers equation in the inviscid limit}.
It is obvious that when $\lambda=1$ (\ref{eqn:Gm1/2}) is
reduced to (\ref{eqn:Gwt}) and hence,
if $\rw_0 = \rw_{1, 0}$, then
\begin{gather}
\rw_t = \rw_{1, t}, \qquad t \geq 0.
\label{eqn:rw_1t}
\end{gather}
In other words, the process $(\rw_{\lambda, t})_{t \geq 0}$ is
the one-parameter ($\lambda \in \R_+$) extension
of $(\rw_t)_{t \geq 0}$.

We denote by $\1_{(E)}$ the indicator function
of an event $E$;
$\1_{(E)}=1$ if $E$ occurs,
and $\1_{(E)}=0$ otherwise.
As a special case, Kronecker's delta is defined by
$\delta^{ij}:=\1_{(i=j)}$, $i, j \in \N$.
Moreover, the $\sigma$-algebra
of Borel sets on $S$ is denoted by $\cB(S)$
and we define $1_B(x) := \1_{(x \in B)}$
for $B \in \cB(S)$.

For $p>0$ and $\mu \in \cP^0(\R)$, the $p$-th push-forward
measure $\mu^{(p)}$ is defined by \cite{PAS12}
\begin{gather}
\mu^{(p)}(B) :=\int_{\R} 1_{B}(|x|^p) \mu({\rm d}x),
\qquad B \in \cB((0, \infty)).
\label{eqn:mu_p}
\end{gather}
We define
\begin{gather}
\rmm_{\lambda, t} := \rw_{\lambda, t}^{(2)}
\in \cP^0(\R_+), \qquad t \geq 0,
\label{eqn:def_rmm}
\end{gather}
provided the matching of initial measures
$\rmm_{\lambda, 0}=\rw_{\lambda, 0}^{(2)}$.
Note that combining this definition with~(\ref{eqn:rw_1t}) we have
the equality
\begin{gather}
\rw_t^{(2)}=\rmm_{1,t}, \qquad t \geq 0.
\label{eqn:rw_t2_rm1t}
\end{gather}
We can show that by (\ref{eqn:def_rmm})
the PDE (\ref{eqn:Gm1/2}) is transformed
to the following equation for the Cauchy transform
$G_{\rmm_{\lambda, t}}(z)$ of
$\rmm_{\lambda, t}$, $\lambda \in \R_+$,
\begin{gather}
\frac{\partial G_{\rmm_{\lambda,t}}(z)}{\partial t}
+ \big\{2z G_{\rmm_{\lambda,t}}(z) - (1-\lambda) \big\}
\frac{\partial G_{\rmm_{\lambda,t}}(z)}{\partial z}
+G_{\rmm_{\lambda,t}}(z)^2 =0,
\qquad t \geq 0.
\label{eqn:Gm}
\end{gather}
It should be noted that this PDE
is obtained when we consider the hydrodynamic limit of
the system of SDEs \cite{CDG01} known as
the \textit{Bru--Wishart process}
in multivariate stochastic calculus \cite{Bru91,KT04}
and as the \textit{Laguerre process}
in dynamical random matrix theory \cite{KO01}.
The PDEs (\ref{eqn:Gm1/2}) and (\ref{eqn:Gm}) describe
the large-number limits of colors of the systems
with the \textit{chiral symmetry}
in the \textit{quantum chromodynamics} (QCD)
in high energy physics
\cite{BN08,BNW13b,JNPZ99,LWZ16,Neu08a,Neu08b,SV93,Ver94,VZ93}.
(See also \cite{ABD11,For10}.)
We call (\ref{eqn:Gm1/2}) and (\ref{eqn:Gm})
\textit{complex Burgers-type equations}
\cite{BNW13,BNW14,EK20,FG16},
in which drift terms are modified and
``external-force terms'' are added compared with
(\ref{eqn:Gwt}).
(See Section \ref{sec:hydro} below.)

The simplest initial probability measure
in $\cP^0(S)$, $S=\R$ or $\R_+$, is
the single delta measure~$\delta_0$ at the origin.
We regard the solution of the complex
Burgers-type equation starting from
$G_{\delta_0}(z)=1/z$ as the \textit{fundamental solution}
and denote the obtained measure-valued process
as $\big(\mu_t^0\big)_{t \geq 0}$ with a superscript 0.
We can show that (see, for instance, \cite{MS17})
\begin{gather*}
G_{\rw^0_t}(z)=\frac{1}{2t} \Big[z - \sqrt{z^2-4t}\Big], \qquad t \geq 0,
\\
G_{\rmm^0_{\lambda,t}}(z) =
\frac{1}{2tz} \Big[z \!+t(1\!-\lambda) \!-\sqrt{(z\!-x^+_{\lambda,t})(z\!-x^-_{\lambda,t})}\,\Big]
\qquad \text{with}\quad
x^{\pm}_{\lambda,t}= t \big(1 \pm \sqrt{\lambda}\big)^2, \quad t \geq 0,
\end{gather*}
and they determine the time-dependent measures
for $t \geq 0$ as
\begin{gather}
\rw^0_t({\rm d}x) =
\frac{1}{2 \pi t} \sqrt{4t-x^2}
\, 1_{[-2 \sqrt{t}, 2 \sqrt{t}]}(x)\, {\rm d}x,
\label{eqn:mu_B0_1}
\\
\rmm^0_{\lambda,t}({\rm d}x) =
\max(0, 1-\lambda) \delta_0({\rm d}x)+
\frac{1}{2 \pi t x} \sqrt{(x-x^{-}_{\lambda,t})
(x^{+}_{\lambda,t}-x)}
\, 1_{[x^{-}_{\lambda,t}, x^{+}_{\lambda,t}]}(x)\, {\rm d}x,
\label{eqn:mu_W0_1}
\end{gather}
respectively.
The measure (\ref{eqn:mu_B0_1})
is known as the centered \textit{Wigner's semicircle distribution}
with variance $t$
and (\ref{eqn:mu_W0_1}) is as
the \textit{two-parametric Marcenko--Pastur distribution}
with parameters $\lambda$ and $t$ \cite{BNW13,BNW14,EK20}.
For the fundamental solutions explicitly given by (\ref{eqn:mu_B0_1})
and (\ref{eqn:mu_W0_1}),
it is easy to verify the well-known equality
$\big(\rw_t^0\big)^{(2)}=\rmm^0_{1,t}$, $t \geq 0$.
The equality~(\ref{eqn:rw_t2_rm1t}) mentioned above
generalizes it for any initial probability measure
with bounded support satisfying
$\rw_0^{(2)}=\rmm_{1, 0}$.

Let $\cP^0_{\rs}(\R)$ be the set of all symmetric Borel
probability measures on $\R$
(i.e., $\mu(B)=\mu(-B)$, $B \in \cB((0, \infty))$
for $\mu \in \cP^0_{\rs}(\R)$)
and define the
\textit{symmetric Bernoulli delta measure}
with displacement $2a >0$ as
\begin{gather}
\rd_{a} := \frac{1}{2} (\delta_{-a}+\delta_a)
\in \cP^0_{\rs}(\R).
\label{eqn:da}
\end{gather}
The processes $(\rw_t)_{t \geq 0}$
and $(\rw_{\lambda, t})_{t \geq 0}$
starting from $\rd_a$, which are defined
as the time-dependent probability measures
so that their Cauchy transforms solve
the PDEs (\ref{eqn:Gwt}) and (\ref{eqn:Gm1/2})
under the initial condition $G_{\rd_a}(z)=z/\big(z^2-a^2\big)$,
were reported in \cite{AK16,Nad11,War14}
and in \cite{EK20}, respectively.
These solutions show a dynamical phase transition, in which
the positive parameter $a$ controls the transition observed at a~critical time.
The singularity associated with this phase transition
is very interesting and important,
since it gives the mean-field description
of the spontaneous chiral symmetry breaking \cite{EK20}.
But, due to this singularity, the solutions are
much more complicated compared with
the fundamental solutions (\ref{eqn:mu_B0_1})
and $(\rw_{\lambda, t}^0)_{t \geq 0}$ obtained from
(\ref{eqn:mu_W0_1}) by (\ref{eqn:def_rmm})
(see \cite[Remark~3]{EK20}).
It had seemed to be difficult
to argue general properties of these measure-valued processes.
(See \cite{GMS21} for the exact solutions for other initial
probability measures.)

In the present paper, we solve the
\textit{initial-value problem}
for these measure-valued processes
$(\mu_t)_{t \geq 0}=(\rw_t)_{t \geq 0}$,
$(\rw_{\lambda, t})_{t \geq 0}$, and
$(\rmm_{\lambda, t})_{t \geq 0}$.
We found that, by the method of characteristic curves,
each of the PDEs of the Cauchy transforms
$(G_{\mu_t})_{t \geq 0}$, given by
(\ref{eqn:Gwt}), (\ref{eqn:Gm1/2}), and (\ref{eqn:Gm}),
is transformed to a \textit{functional equation}
which relates $G_{\mu_t}(z)$ at an arbitrary time $t > 0$
with the initial function $G_{\mu_0}(z)$.
The result for $(G_{\rw_t})_{t \geq 0}$ given as
Proposition \ref{thm:GD}$(i)$
is well known and found in literature.
The results for $(G_{\rmm_{\lambda, t}})_{t \geq 0}$
and $(G_{\rw_{\lambda, t}})_{t \geq 0}$ given as
Proposition \ref{thm:GD}$(ii)$ and~$(iii)$, respectively, are new,
but they are complicated and do not seem to be useful
for explicit calculations.
On the other hand, in free probability theory
\cite{AHS13,BV93,Bia97,Bia97b,HP00,MS17,NS06,PAS12,Voi86},
we learn other importance transformations
of probability measures different from the Cauchy transform;
the $R$-transform and the $S$-transform.
They define new types of convolutions of
probability measures called
\textit{free convolutions}.
(A brief review is given shortly.)
We applied these transformations
to our functional equations.
The results for the $R$-transforms denoted by $(R_{\mu_t})_{t \geq 0}$
and the $S$-transforms by $(S_{\mu_t})_{t \geq 0}$
are given by Theorems \ref{thm:main}
and \ref{thm:main2}, respectively.
The obtained functional equations are
expressed using the $R$-transforms and the $S$-transforms
of the fundamental solutions and much simplified.
In particular, as shown in Theorem \ref{thm:main}$(i)$,
the functional equation (\ref{eqn:Burgers_R_1}) implies
the decomposition formula (\ref{eqn:Burgers_R_2})
with respect to the free additive convolution
(see also Remarks~\ref{Remark1} and \ref{Remark2}).
Theorem \ref{thm:main2}$(i)$ gives the simple but
useful relationship among $(\rw_{\lambda, t})_{t \geq 0}$,
$(\rmm_{\lambda, t})_{t \geq 0}$, and $\rd$
(see Remark~\ref{Remark4}).
The functional equations are also useful
for explicit calculations
as demonstrated by
Propositions \ref{thm:wa} and \ref{thm:ma}.
There the solutions for the processes
$(\rw_{\lambda, t})_{t \geq 0}$ starting from
$\rd_a$, $a >0$ and $(\rmm_{\lambda, t})_{t \geq 0}$
from $\delta_b$, $b >0$ are shown,
which are much simpler than the corresponding
solutions of the Cauchy transformations
reported in \cite{AK16,Nad11,War14} and
\cite{EK20}.
Comparing the results for $(\rw_{\lambda, t})_{t \geq 0}$
and those for $(\rmm_{\lambda, t})_{t \geq 0}$,
the latter is simpler than the former.
In the original complex Burgers-type equations,
it is obvious at the PDE (\ref{eqn:Gm1/2})
for $(\rw_{\lambda, t})_{t \geq 0}$ that $(\rw_{\lambda,t})_{t \geq 0}$
is a one-parameter extension $(\lambda \in \R_+)$
of $(\rw_t)_{t \geq 0}$, but it is not at the PDE
(\ref{eqn:Gm}) for $(\rmm_{\lambda, t})_{t \geq 0}$.
Our results shows, however, that the transformation from
$(\rw_{\lambda, t})_{t \geq 0}$ to $(\rmm_{\lambda,t})_{t \geq 0}$
by the second push-forward transform
is effective to make practical problems be simpler
and solvable.
Further study of the $p$-th push-forward transforms
(\ref{eqn:mu_p}) will be important.

Given a probability measure
$\mu \in \cP^0(S)$, let
$\tau_n(\mu)$ represent
the $n$-th moment of $\mu$,
$n \in \N$ and define the \textit{moment generating function} by
\begin{gather}
\Psi_{\mu}(z) := \int_S \frac{xz}{1-xz} \mu({\rm d}x)
=\sum_{n=1}^{\infty} \tau_n(\mu) z^n.
\label{eqn:Psi2}
\end{gather}
The Cauchy transform $G_{\mu}$
of $\mu$ is related to $\Psi_{\mu}$ by
\begin{gather}
\Psi_{\mu}(z)=\frac{G_{\mu}(1/z)}{z}-1 \Longleftrightarrow
G_{\mu}(1/z)=z (\Psi_{\mu}(z)+1).
\label{eqn:Psi1}
\end{gather}
We write the inverse function of $G_{\mu}$ as
$G_{\mu}^{\bra -1 \ket}$.
The \textit{$R$-transform} of
$\mu$ is then defined by
\begin{gather}
R_{\mu}(z) := z G_{\mu}^{\bra -1 \ket}(z)-1
\Longleftrightarrow
G_{\mu}^{\bra -1 \ket}(z)=\frac{R_{\mu}(z)+1}{z}.
\label{eqn:R1}
\end{gather}
This is the generating function of
\textit{free cumulants} $\kappa_n=\kappa_n(\mu)$,
$n \in \N$,
\begin{gather*}
R_{\mu}(z)=\sum_{n=1}^{\infty} \kappa_n(\mu) z^n.
\end{gather*}
(Notice that in the free probability literature
the function $\cR_{\mu}(z):=R_{\mu}(z)/z$
is also used and referred to as ``the $R$-transform
of $\mu$''. See, for instance, \cite[equation~(16.8)]{NS06}.)
The relations between the moments $\{\tau_n(\mu)\}_{n \in \N}$
in (\ref{eqn:Psi2}) and the
free cumulants $\{\kappa_n(\mu)\}_{n \in \N}$ are given by
\begin{gather*}
\kappa_1(\mu) = \tau_1(\mu),
\\
\kappa_2(\mu)= \tau_2(\mu) - \tau_1(\mu)^2,
\\
\kappa_3(\mu) = \tau_3(\mu)-3 \tau_2(\mu) \tau_1(\mu)+ 2 \tau_1(\mu)^3,
\\
\kappa_4(\mu) = \tau_4(\mu) - 4 \tau_3(\mu) \tau_1(\mu)-2 \tau_2(\mu)^2
+ 10 \tau_2(\mu) \tau_1(\mu)^2-5 \tau_1(\mu)^4,
\\
\dots\dots\dots\dots,
\end{gather*}
which are generally different from the relations
satisfied in the
``classical probability theory''.
For two probability measures
$\mu$, $\nu$,
the \textit{free additive convolution} of
them is denoted by
$\mu \boxplus \nu$ and defined by \cite{BV93}
\begin{gather*}
R_{\mu \boxplus \nu}(z)=R_{\mu}(z)+R_{\nu}(z),
\end{gather*}
which implies\vspace{-.5ex}
\begin{gather*}
\kappa_{n}(\mu \boxplus \nu)= \kappa_{n}(\mu) + \kappa_{n}(\nu),
\qquad n \in \N.
\end{gather*}
For $\mu \in \cP^0(\R_+)$
with $\mu(\{0\}) < 1$,
the moment
generating function $\Psi_{\mu}(z)$ defined by
(\ref{eqn:Psi1}) has a unique inverse\vspace{-.5ex}
\begin{gather}
\chi_{\mu}(z) := \Psi_{\mu}^{\bra -1 \ket}(z)
\label{eqn:chi_Psi1}
\end{gather}
on the left-half plane $\sqrt{-1} \C^+$ \cite{BV93}.
In this case
the \textit{$S$-transform} of $\mu$ is
defined by\vspace{-.5ex}
\begin{gather}
S_{\mu}(z) :=\frac{1+z}{z} \chi_{\mu}(z),
\qquad z \in \Psi_{\mu}\big(\sqrt{-1} \C^+\big).
\label{eqn:S_def1}
\end{gather}
For two probability measures $\mu$, $\nu$
having $S$-transforms,
the \textit{free multiplicative convolution} \cite{BV93}
is defined as
the the probability measure $\mu \boxtimes \nu$
such that\vspace{-.5ex}
\begin{gather*}
S_{\mu \boxtimes \nu}(z)
=S_{\mu}(z) S_{\nu}(z)
\end{gather*}
for $z$ in a common region of
$\Psi_{\mu}\big(\sqrt{-1} \C^{+}\big) \cup \Psi_{\nu}\big(\sqrt{-1} \C^+\big)$.

The definition of the $S$-transform can be extended
for symmetric probability measures
$\mu \in \cP^0_{\rs}(\R)$ \cite{BV93,PAS12}.
When $\mu \in \cP^0_{\rs}(\R)$ with $\mu(\{0\}) < 1$,
$\Psi_{\mu}(z)$ has a unique inverse on
$H:=\{z \in \C^{-}\colon |\operatorname{Re} z| < |\operatorname{Im} z|\}$,
$\chi_{\mu} \colon \Psi_{\mu}(H) \to H$ and
a unique inverse on
$\widetilde{H}:=\{z \in \C^{+}\colon |\operatorname{Re} z| < |\operatorname{Im} z|\}$,
$\widetilde{\chi}_{\mu} \colon \Psi_{\mu}(\widetilde{H}) \to \widetilde{H}$,
where $\C^- :=\{z \in \C\colon \operatorname{Im} z <0\}$.
Therefore, there are two $S$-transforms for $\mu$
given by\vspace{-.5ex}
\begin{gather*}
S_{\mu}(z)=\frac{1+z}{z} \chi_{\mu}(z)
\qquad \text{and} \qquad
\widetilde{S}_{\mu}(z)=\frac{1+z}{z} \widetilde{\chi}_{\mu}(z),
\end{gather*}
and they satisfy\vspace{-.5ex}
\begin{gather*}
S_{\mu}(z)^2 = \frac{1+z}{z} S_{\mu^{(2)}}(z)
\qquad \text{and} \qquad
\widetilde{S}_{\mu}(z)^2 = \frac{1+z}{z} S_{\mu^{(2)}}(z).
\end{gather*}
It is known that
for a probability measure $\mu \in \cP^0(\R_+)$,
there exists a unique symmetric probability measure
$\mu^{\bf s} \in \cP^0_{\rs}(\R)$ such that\vspace{-.5ex}
\begin{gather*}
\int_{\R} f\big(x^{2}\big) \mu^{\bf s}(\rd x) = \int_{\R} f(x) \mu^{(1/2)}(\rd x)
\end{gather*}
for every compactly supported continuous function $f$.
The probability measure $\mu^{\bf s}$ is called
\textit{symmetrization} of a probability measure $\mu$.
For details, see \cite[p.~134]{HiMo2010}.
We note that the $S$-transform of
$\rd_a$ defined by (\ref{eqn:da}) is given by
\begin{gather}
S_{\rd_a}(z)=\frac{1}{a} \sqrt{\frac{1+z}{z}},
\qquad a > 0.
\label{eqn:Sda}
\end{gather}
If the parameter $a=1$, we use the notation
$\rd$ for $\rd_1$ \cite{PAS12}.

The PDFs (\ref{eqn:Gwt}), (\ref{eqn:Gm1/2})
and (\ref{eqn:Gm}) for the Cauchy transforms
$(G_{\mu_t})_{t \geq 0}$
are transforms as follows:
For $(R_{\mu_t})_{t \geq 0}$,
\begin{gather*}
\frac{\partial R_{\rw_t}(z)}{\partial t} - z^2 =0,
\\
\frac{\partial R_{\rw_{\lambda, t}}(z)}{\partial t}
- \frac{(1-\lambda)z^3}{2 (R_{\rw_{\lambda, t}}(z)+1)^2}
\frac{\partial R_{\rw_{\lambda, t}}(z)}{\partial z}-z^2 +
\frac{(1-\lambda) z^2}{R_{\rw_{\lambda, t}}(z)+1}=0,
\\
\frac{\partial R_{\rmm_{\lambda, t}}(z)}{\partial t}-z^2
\frac{\partial R_{\rmm_{\lambda, t}}(z)}{\partial z}
-z (R_{\rmm_{\lambda, t}}(z)+\lambda)=0,
\end{gather*}
and for $(S_{\mu_t})_{t \geq 0}$,
\begin{gather*}
\frac{\partial S_{\rw_t}(z)}{\partial t}
+z^2 S_{\rw_t}(z)^2 \frac{\partial S_{\rw_t}(z)}{\partial z}
+z S_{\rw_t}(z)^3=0,
\\
\frac{\partial S_{\rw_{\lambda,t}}(z)}{\partial t}
+z^2 \bigg(1-\frac{1-\lambda}{1+z} \bigg)S_{\rw_{\lambda, t}}(z)^2
\frac{\partial S_{\rw_{\lambda, t}}(z)}{\partial z}
+z \bigg\{1 - \frac{(1-\lambda) (z+2)}{2(1+z)^2}\bigg\}
S_{\rw_{\lambda, t}}(z)^3=0,
\\
\frac{\partial S_{\rmm_{\lambda,t}}(z)}{\partial t}
+z(z+\lambda) S_{\rmm_{\lambda, t}}(z)
\frac{\partial S_{\rmm_{\lambda, t}}(z)}{\partial z}+(2z+\lambda)
S_{\rmm_{\lambda, t}}(z)^2=0.
\end{gather*}
We find that the equation of $(R_{\rw_t})_{t \geq 0}$ is
extremely simple and it corresponds to the asymptotic freeness
of a random matrix in the GUE and
a deterministic Hermitian matrix
(see Theorem~\ref{thm:main}$(i)$, Remarks~\ref{Remark1} and \ref{Remark2} below).
On the other hand, the ``external-force terms'' include
both of the coordinate $z$ and the ``fields'' in other equations
and seem to be more complicated than the equation
(\ref{eqn:Gm}) for $(G_{\rmm_{\lambda, t}})_{t \geq 0}$
whose external-force term
is simply given by $G_{m_{\lambda, t}}(z)^2$.
Theorems \ref{thm:main} and \ref{thm:main2} given below
solve the initial-value problems for all of them.

\subsection[Results for the R-transforms]{Results for the $\boldsymbol R$-transforms}

The following are the results
for the $R$-transforms.

\begin{Theorem}\label{thm:main}\quad
\begin{enumerate}\itemsep=0pt
\item[$(i)$]
Assume that $\rw_0 \in \cP^0(\R)$. Then,
\begin{gather}
R_{\rw_t}(z)=R_{\rw_0}(z)+R_{\rw^0_t}(z), \qquad t \geq 0,
\label{eqn:Burgers_R_1}
\end{gather}
where
$R_{\rw_t^0}(z)=t z^2$, $t \geq 0$.
It means that
$(\rw_t)_{t \geq 0}$ is decomposed into
the initial probability measure $\rw_0$
and the fundamental solution
$\big(\rw_t^0\big)_{t \geq 0}$ with respect to the
free additive convolution as follows,
\begin{gather}
\rw_t=\rw_0 \boxplus \rw^0_t, \qquad
t \geq 0.
\label{eqn:Burgers_R_2}
\end{gather}
\item[$(ii)$]
Assume that $\rmm_{\lambda, 0} \in \cP^0(\R_+)$,
$\lambda \in \R_+$. Then,
\begin{gather}
R_{\rmm_{\lambda, t}}(z)
=\frac{1}{1-t z} R_{\rmm_{\lambda, 0}} \bigg( \frac{z}{1-tz} \bigg)
+ R_{\rmm^0_{\lambda, t}}(z),
\qquad t \geq 0,
\label{eqn:W_Burgers_R_1}
\end{gather}
where
$R_{\rmm_{\lambda, t}^0}(z)=\lambda t z/(1-tz)$,
$t \geq 0$.
\item[$(iii)$]
Assume that $\rw_{\lambda, 0} \in \cP^0_{\rs}(\R)$,
$\lambda \in \R_+$. Then,
\begin{gather}
R_{\rw_{\lambda, t}}(z)
+ \frac{(1-\lambda) t z^2}{R_{\rw_{\lambda, t}}(z)+1}\nonumber
\\ \qquad
{}= R_{\rw_{\lambda, t}^0}(z)
+ \frac{(1-\lambda) t z^2}{R_{\rw_{\lambda, t}^0}(z)+1}\nonumber
\\
 \qquad \hphantom{=}
{}+R_{\rw_{\lambda, 0}} \bigg(
z \sqrt{ 1 \!- (1\!-\lambda)
\bigg\{ \frac{1}{R_{\rw_{\lambda, t}}(z)\!+1}
-\frac{1}{R_{\rw_{\lambda, t}}(z)\!+1\!-tz^2} \bigg\}
}\bigg),\quad t \geq 0,\!\!
\label{eqn:chiral_Burgers_R_1}
\end{gather}
where
$R_{\rw_{\lambda, t}^0}(z)
=\big\{{-}1+tz^2+\sqrt{1+2(2\lambda-1)tz^2+t^2 z^4}\big\}/2$.
\end{enumerate}
\end{Theorem}

\begin{Remark}\label{Remark1}
 Consider a matrix-valued Brownian motion
 $(M_t)_{t \geq 0}$ which is given by
 a time-evolution of a Hermitian
 $N \times N$ matrix starting from
 a Hermitian matrix $M_0$.
 When $M_0$ is a~null matrix,
 we write this process as $\big(M_t^0\big)_{t \geq 0}$.
 Then we have
 \begin{gather}
 M_t=M_0+M_t^0, \qquad t \geq 0. \label{eqn:RM_process}
 \end{gather}
 We assume that the empirical eigenvalue distribution
 of $M_0$ converges to $\rw_0$ as $N \to \infty$.
 As~an eigenvalue process of (\ref{eqn:RM_process}),
 we can obtain Dyson's Brownian motion model with $\beta=2$
 starting from the eigenvalues of $M_0$.
 (See Section \ref{sec:hydro} below.)
 Moreover, for any $\beta >0$, we can
 obtain the same Cauchy transform.
 The process $(\rw_t)_{t \geq 0}$ is obtained as
 the time-evolution of the
 limit empirical measure of
 Dyson's Brownian motion model.
 We can show that
 $M_0$ and $M_t^0$ are \textit{asymptotically free}
 for any $t \geq 0$.
 Thus at each time $t \geq 0$,
 the limiting eigenvalue distribution~$\rw_t$ of $M_t$
 converges to $\rw_0 \boxplus \rw_t^0$.
 That is, the assertion $(i)$ of Theorem \ref{thm:main}
 is consistent with the asymptotic freeness of
 a random matrix in the GUE and an arbitrary
 deterministic Hermitian matrix.
 Such an interpretation of the assertion $(ii)$ of
 Theorem \ref{thm:main} is not yet known
 and is left as a challenging future problem.
\end{Remark}


\begin{Remark}
 \label{Remark2}
 With respect to the process $(\rw_t)_{t \geq 0}$,
 it is pointed out that
 the functional equations~(\ref{eqn:Burgers_R_1})
 for the $R$-transform
 and (\ref{eqn:GD_func}) for the Cauchy transform
 given below are consequences of the
 \textit{Markov property of freeness} and its relation to
 \textit{analytic subordination} \cite{BB07,Bia98,Voi00}.
 As shown in Sections \ref{sec:solution} and \ref{sec:proof_main},
 we will prove (\ref{eqn:Burgers_R_1}) from (\ref{eqn:GD_func}).
 In the context of analytic subordination,
 the functional equation (\ref{eqn:GD_func}) shall be considered
 to prove unique existence of the one-parameter
 ($t \in [0, \infty)$) family of functions
 $\omega_t$ such that
 \begin{gather*}
 \lim_{y \uparrow \infty} \frac{\omega_t\big(\sqrt{-1} y\big)}{\sqrt{-1} y}
 =1, \qquad \text{and}
 \qquad G_{\rw_t} = G_{\rw_0} \circ \omega_t,
 \quad t \in [0, \infty),
 \end{gather*}
 by explicitly showing that
 $\omega_t(z)=z-t G_{\rw_t}(z)$.
 Here the unique existence of $(G_{\rw_t}(z))_{t \geq 0}$
 with appropriate properties is guaranteed by the fact that
 $(G_{\rw_t}(z))_{t \geq 0}$ solves the complex
 Burgers equation (\ref{eqn:Gwt}).
 For the other processes $(\rmm_{\lambda, t})_{t \geq 0}$
 and $(\rw_{\lambda, t})_{t \geq 0}$,
 the connections of our results
 with analytic subordination properties should be studied
 in the future.
\end{Remark}

\begin{Remark} \label{Remark3}
 The process $\big(\rw_t^0\big)_{t \geq 0}$ given by (\ref{eqn:mu_B0_1})
 is identified with the \textit{free Brownian motion}
 studied in free probability theory \cite{Bia97,Spe90}.
 Here we would like to consider
 the process $(\rw_t)_{t \geq 0}$ determined by
 the Burgers equation~(\ref{eqn:Gwt})
 as a generalization of the free Brownian motion,
 since initial probability measure $\rw_0$ is
 now arbitrary in $\cP^0(\R)$.
 See also~\cite{Bia97b,Voi86} for the important connection
 between the complex Burgers equation~(\ref{eqn:Gwt})
 and free probability theory.
 Capitaine and Donati-Martin \cite{CDM05}
 introduce the \textit{free Wishart processes}
 based on the two-parametric Marcenko--Pastur
 distribution (\ref{eqn:mu_W0_1}).
 Our process $(\rmm_{\lambda, t})_{t \geq 0}$
 is defined as a solution of the complex Burgers-type
 equation (\ref{eqn:Gm}) specified by its initial probability measure
 $\rmm_{\lambda, 0} \in \cP^0(\R_+)$.
 In~other words, $(\rmm_{\lambda, t})_{t \geq 0}$ is a family
 of processes parameterized by an initial measure
 $\rmm_{\lambda, 0}$, and hence
 it is different from the free Wishart process.
 In the functional equation (\ref{eqn:W_Burgers_R_1})
 in Theorem \ref{thm:main}$(ii)$,
 the parameterization by an initial probability measure
 is realized by the first term in the r.h.s.,
 which is added to the fundamental solution
 for the $R$-transform, $R_{\rmm_{\lambda, t}^0}(z)$,
 of the two-parametric Marcenko--Pastur
 distribution (\ref{eqn:mu_W0_1}).
 The \textit{rectangular free convolutions} studied by
 Benaych-Georges \cite{BenGeo09,BenGeo10} are
 very interesting and important extensions of the
 square free convolutions. They are based on the
 original Marcenko--Pastur distribution
 (i.e., the special case of~(\ref{eqn:mu_W0_1}) at $t=1$).
 By this reason, it is not easy to discuss the
 present study from the viewpoint of the
 rectangular free convolutions.
 We want to leave this topic as a future problem.
 The equality~(\ref{eqn:chiral_Burgers_R_1})
 in Theorem \ref{thm:main}$(iii)$
 seems to be so complicated, but it clearly shows that
 if $\lambda=1$, this equality is reduced to
 (\ref{eqn:Burgers_R_1})
 as expected from (\ref{eqn:rw_1t}).
 To the best of our knowledge,
 the chiral GUE and its time evolution $(\rw_{\lambda, t})_{t \geq 0}$
 determined by (\ref{eqn:Gm1/2}) have not been systematically
 studied in free probability theory.
\end{Remark}

\subsection[Results for the S-transforms]{Results for the $\boldsymbol S$-transforms}

For $\big(\rw^0_t\big)_{t \geq 0}$ starting from $\delta_0$,
it is easy to verify that (see, for instance, \cite{PAS12})
\begin{gather}
S_{\rw^0_t}(z)=\sqrt{\frac{1}{tz}}, \qquad t \geq 0.\label{eqn:S_wt}
\end{gather}
For the process $\big(\rmm^0_{\lambda,t}\big)_{t \geq 0}$
starting from $\delta_0$, we have~\cite{PAS12}
\begin{gather*}
S_{\rmm^0_{\lambda,t}}(z)=\frac{1}{t(z+\lambda)}, \qquad t \geq 0.
\end{gather*}

\begin{Theorem}\label{thm:main2}\quad
\begin{enumerate}\itemsep=0pt
\item[$(i)$]
Provided that
$\rw_{\lambda,0}=\rmm_{\lambda, 0}^{\bf s}$, $\lambda \in \R_+$,
the following equality holds,
\begin{gather}
S_{\rw_{\lambda, t}}(z)=S_{\rd}(z) \sqrt{S_{\rmm_{\lambda, t}}(z)},
\qquad t \geq 0.
\label{eqn:S_multiple}
\end{gather}
\item[$(ii)$]
Assume that the $S$-transform of the initial probability measure
$S_{\rw_0}(z)$ is well defined.
Then,
\begin{gather}
\frac{S_{\rw_t}(z)}{1-(S_{\rw_t}(z)/S_{\rw^0_t}(z))^2}
=S_{\rw_0} \big(z \big\{ 1-\big(S_{\rw_t}(z)/S_{\rw^0_t}(z)\big)^2 \big\} \big),
\qquad t \geq 0.\label{eqn:S_time_Dyson}
\end{gather}
\item[$(iii)$]
Assume that the $S$-transform of the initial measure
$S_{\rmm_{\lambda, 0}}(z)$, $\lambda \in \R_+$ is well defined.
Then,
\begin{gather}
\frac{S_{\rmm_{\lambda, t}}(z)}
{\big\{1-S_{\rmm_{\lambda, t}}(z)/S_{\rw^0_t}(z)^2\big\}
\big\{1-S_{\rmm_{\lambda, t}}(z)/S_{\rmm^0_{\lambda, t}}(z) \big\}}\nonumber
\\ \qquad
{}=S_{\rmm_{\lambda, 0}} \big(z \big\{ 1-S_{\rmm_{\lambda, t}}(z)/
S_{\rmm^0_{\lambda, t}}(z) \big\} \big),
\qquad t \geq 0.
\label{eqn:S_time_Wishart}
\end{gather}
\item[$(iv)$]
Assume that the $S$-transform of the initial measure
$S_{\rw_{\lambda, 0}}(z)$, $\lambda \in \R_+$ is well defined.
Then,
\begin{gather}
\sqrt{\frac{1- \frac{z}{1+z}
\big(S_{\rw_{\lambda, t}}(z)/S_{\rw_{\lambda,t}^0}(z)\big)^2}
{1- \frac{z}{1+z} \big(S_{\rw_{\lambda, t}}(z)/S_{\rw_t^0}(z)\big)^2}}
\frac{S_{\rw_{\lambda, t}}(z)}
{1-\big(S_{\rw_{\lambda, t}}(z)/S_{\rw_{\lambda, t}^0}(z)\big)^2}\nonumber
\\ \qquad
{}=S_{\rw_{\lambda, 0}} \big(z \big\{1- (S_{\rw_{\lambda, t}}(z)/S_{\rw_{\lambda, t}^0}(z))^2\big\} \big),
\qquad t \geq 0.
\label{eqn:S_time_sqrWishart}
\end{gather}
Since $\rw_{1,t}=\rw_t$, if $\lambda=1$
\eqref{eqn:S_time_sqrWishart} is reduced to~\eqref{eqn:S_time_Dyson}.
\end{enumerate}
\end{Theorem}

\begin{Remark} \label{Remark4}
 The equality (\ref{eqn:S_multiple}) in
 the assertion $(i)$ of Theorem \ref{thm:main2}
 can be regarded as a ``push-back'' representation
 of the equality (\ref{eqn:def_rmm}) expressed
 using the second push-forward measure.
 This is very simple, but reveals an important role
 of the symmetric Bernoulli delta measure~$\rd_a$
 with $a=1$ defined by (\ref{eqn:da}),
 whose $S$-transform is given by (\ref{eqn:Sda}).
 If $\sqrt{S_{\rmm_{\lambda, t}}(z)}$, $t \geq 0$ is realized as
 an $S$-transform of a probability measure,
 say $(\nu_{\lambda, t})_{t \geq 0}$,
 then (\ref{eqn:S_multiple}) implies
 the equality
 $\rw_{\lambda, t}= \rd \boxtimes \nu_{\lambda, t}$, $t \geq 0$.
\end{Remark}

\subsection{Applications}

First we apply the above theorems to
the process $(\rw_t)_{t \geq 0}$
starting from the symmetric Bernoulli delta measure $\rd_a$
with displacement $2a >0$ given by (\ref{eqn:da}).
Here we write this process as~$(\rw_t^a)_{t \geq 0}$.
\begin{Proposition}
\label{thm:wa}
The $R$-transform and the $S$-transform of
$(\rw_t^a)_{t \geq 0}$ are given by
\begin{gather}
R_{\rw_t^a}(z) =\frac{1}{2} \Big[ 2 t z^2
-1 + \sqrt{1+ 4 a^2 z^2} \Big], \qquad t \geq 0,
\label{eqn:R_wa}
\\
S_{\rw_t^a}(z)= \frac{1}{(t z)^{1/2}}
\bigg[ 1 + \frac{1}{2z} + \frac{a^2}{2 t z}
- \frac{1}{2z} \bigg( 1+\frac{a^2}{t} \bigg)
\sqrt{1+ \frac{4a^2/t}{(1+a^2/t)^2} z } \, \bigg]^{1/2},
\qquad t \geq 0.
\label{eqn:S_wa}
\end{gather}
\end{Proposition}

Note that (\ref{eqn:R_wa}) determines the
free cumulants of $(\rw_{t}^a)_{t \geq 0}$ as
\begin{gather*}
\kappa_n(\rw_t^a)=
\begin{cases}
t+a^2, & n=2,
\cr
\displaystyle{
-(-1)^{n/2} \frac{(n-3)!!}{(n/2)!} 2^{n/2-1} a^n
},
& n \in \{4, 6, 8, \dots\},
\cr
0, & \mbox{otherwise},
\end{cases}
\qquad t \geq 0.
\end{gather*}
As well known, for the fundamental solution
$\big(\rw_t^0\big)_{t \geq 0}$,
$R_{\rw_t^0}(z)=t z^2$, and hence
$\kappa_n\big(\rw_t^0\big)=t \delta_{n 2}$,
$n \in \N$, $t \geq 0$.
The complexity of the solution
$(\rw_t^a)_{t \geq 0}$ with $a >0$
reported in \cite{AK16,Nad11,War14} is simply expressed here
by the emergence of free cumulants $\kappa_n(\rw_t^a)$
for all $n \in \N$, $t \geq 0$.

Next we apply the theorems to
the process $(\rmm_{\lambda, t})_{t \geq 0}$
starting from $\delta_b$ with $b > 0$.
Here we write this process as
$\big(\rmm_{\lambda, t}^b\big)_{t \geq 0}$.
If the variable $x$ is replaced by $x/\lambda$,
the parameters $r$ by $1/\lambda$, and
$a$ by $b/\lambda$
in the \textit{three parametric Marcenko--Pastur measure}
studied in~\cite{EK20},
we obtain the present probability measure
$\rmm_{\lambda, t}^b$, $t \geq 0$.

\begin{Proposition}\label{thm:ma}
The $R$-transform and the $S$-transform of
$\big(\rmm_{\lambda, t}^b\big)_{t \geq 0}$ are given by
\begin{gather}
R_{\rmm_{\lambda, t}^b}(z) =
\frac{z \big\{(\lambda t+b)-\lambda t^2 z \big\}}{(1-tz)^2}, \qquad t \geq 0,
\label{eqn:R_ma}
\\
S_{\rmm_{\lambda, t}^b}(z) \frac{1}{t(\lambda+z)}
\bigg[1+\frac{\lambda}{2z} + \frac{b}{2tz}
- \frac{1}{2z}\bigg( \lambda + \frac{b}{t} \bigg)
\sqrt{ 1 + \frac{4 b/t}{(\lambda+b/t)^2}z } \bigg],
\qquad t \geq 0.
\label{eqn:S_ma}
\end{gather}
\end{Proposition}

Comparing (\ref{eqn:S_wa}) and (\ref{eqn:S_ma}),
we obtain the equality,
\begin{gather}
S_{\rw_t^a}(z)
= \sqrt{\frac{1+z}{z}} \sqrt{S_{\rmm_{1, t}^{a^2}}(z)}.
\label{eqn:Swa}
\end{gather}
It is readily confirmed by the definition
of the second push-forward measure
(\ref{eqn:mu_p}) with $p=2$ that
$\rd_a^{(2)}=\delta_{a^2}$, $a>0$.
Hence, the matching of initial measures is
established, and as a special case of (\ref{eqn:rw_t2_rm1t}),
$(\rw_t^a)^{(2)}=m_{1, t}^{a^2}$, $t \geq 0$, $a>0$.
Therefore,
(\ref{eqn:Swa}) can be regarded as a special case of
the assertion $(i)$ of Theorem \ref{thm:main2}.

Note that (\ref{eqn:R_ma}) determines the
free cumulants of $\rmm_{\lambda, t}^b$ as
\begin{gather*}
\kappa_n\big(\rmm_{\lambda, t}^b\big)
=(\lambda t + b n) t^{n-1},
\qquad n \in \N, \quad t \geq 0.
\end{gather*}
The complexity of $\big(\rmm_{\lambda, t}^b\big)_{t \geq 0}$
with $b>0$ reported in \cite{EK20} is simply expressed
by a shift $\lambda t \to \lambda t + bn$
in the above formulas for $\kappa_n\big(\rmm_{\lambda, t}^b\big)$,
$n \in \N$, $t \geq 0$.
The dynamical phase transitions
studied by \cite{AK16,EK20,Nad11,War14}
seem to be hidden in the above solutions
for the $R$-transforms and the $S$-transforms.
Extracting the singularity at the transition point
from the above results will be a future problem.

The present paper is organized as follows.
In Section~\ref{sec:preliminaries} we explain how the complex
Burgers-type equations are derived in the hydrodynamic limits of
stochastic log-gases.
Then we prove fundamental relations
between a measure $\mu \in \cP^0_{\rs}(\R)$
and its second push-forward measure
$\mu^{(2)} \in \cP^0(\R_+)$.
Section \ref{sec:solution} is devoted to solving
the present three kinds of complex Burgers equations
(\ref{eqn:Gwt}), (\ref{eqn:Gm1/2}), and (\ref{eqn:Gm})
by the method of characteristic curves \cite{CH62,Del97}.
Proofs of Theorems \ref{thm:main}, \ref{thm:main2}
and Propositions \ref{thm:wa}, \ref{thm:ma} are given in
Section \ref{sec:proof}.
Concluding remarks are given in Section \ref{sec:remarks}.

\section{Preliminaries}
\label{sec:preliminaries}

\subsection[From matrix-valued processes to
complex Burgers-type equations through hydrodynamic limit]
{From matrix-valued processes to
complex Burgers-type equations \\through hydrodynamic limit}
\label{sec:hydro}

For $N \in \N :=\{1,2, \dots\}$,
let $\sH_N$ and $\sU_N$ be the space of
$N \times N$ Hermitian matrices and
the group of $N \times N$ unitary matrices, respectively.
We consider complex-valued
continuous semi-martingale processes
$\big(M^{ij}_t\big)_{t \geq 0}, 1\leq i, j \leq N$
with the condition $\overline{M^{ji}_t}=M^{ij}_t$,
where $\overline{z}$ denotes the
complex conjugate of $z \in \C$,
and define an $\sH_N$-valued process by
$M_t=\big(M^{ij}_t\big)_{1 \leq i, j \leq N}$.
For $S=\R$ or $\R_+$, define
the Weyl chambers as
$\W_N(S) := \big\{ \x=\big(x^1, \dots, x^N\big) \in S^N\colon x^1 < \cdots < x^N\big\}$,
and write their closures as
$\overline{\W_N(S)} = \big\{ \x \in \overline{S}^N \colon
x^1 \leq \cdots \leq x^N\big\}$.
For each $t \geq 0$, there exists
$U_t=\big(U^{ij}_t\big)_{1 \leq i, j \leq N} \in \sU_N$ such that
it diagonalizes $M_t$ as
$U^{\dagger}_t M_t U_t
= {\rm diag}\big(\Lambda^1_t, \dots, \Lambda^N_t\big)$
with the eigenvalues $\{\Lambda^i _t\}_{i=1}^N$ of $M_t$,
where $U^{\dagger}_t$ is the Hermitian conjugate of
$U_t$; $\big(U^{ij}_t\big)^{\dagger}=\overline{U^{ji}_t}$, $1 \leq i, j \leq N$,
and we assume
$\Lambda_t :=\big(\Lambda^1_t, \dots, \Lambda^N_t\big)
\in \overline{\W_N(\R)}$, $t\geq 0$.
For ${\rm d} M_t:=\big({\rm d} M^{ij}_t\big)_{1 \leq i, j \leq N}$, define
a set of quadratic variations,
\begin{gather*}
\Gamma^{ij, k \ell}_t:= \big\langle
\big(U^{\dagger} {\rm d} M U\big)^{ij},
\big(U^{\dagger} {\rm d} M U\big)^{k \ell} \big\rangle_t,
\qquad 1 \leq i, j, k, \ell \leq N, \quad t \geq 0.
\end{gather*}
The following is proved \cite{Bru89,Kat16,KT04}.
See~\cite[Section~4.3]{AGZ10} for details of proof.

\begin{Proposition}\label{thm:Bru}
The eigenvalue process $(\Lambda_t)_{t \geq 0}$ satisfies
the following system of SDEs,
\begin{gather*}
{\rm d} \Lambda^i_t={\rm d} \cM^i_t+ {\rm d}J^i_t, \qquad t \geq 0,
\quad 1 \leq i \leq N,
\end{gather*}
where $(\cM^i_t)_{t \geq 0}, 1 \leq i \leq N$
are martingales with quadratic variations
$\langle {\rm d} \cM^i, {\rm d} \cM^j \rangle_t
= \Gamma^{ii, jj}_t {\rm d}t$, $t \geq 0$,
and $(J^i_t)_{t \geq 0}, 1 \leq i \leq N$ are the
processes with finite variations given by
\begin{gather*}
{\rm d}J^i_t= \sum_{j=1}^N
\frac{\1_{(\Lambda^i_t \not= \Lambda^j_t)}}
{\Lambda^i_t-\Lambda^j_t}
\Gamma^{ij, ji}_t {\rm d}t
+ {\rm d} \Upsilon^i_t.
\end{gather*}
Here ${\rm d} \Upsilon^i_t$ denotes
the finite-variation part of
$\big(U^{\dagger}_t {\rm d} M_t U_t\big)^{ii}$, $t \geq 0$,
$1 \leq i \leq N$.
\end{Proposition}

We will show two basic examples of $M_t \in \sH_N, t \geq 0$
and applications of Proposition~\ref{thm:Bru}, see~\cite{KT04}.
Let $\nu \in \N_0 := \N \cup \{0\}$ and
$\big(B^{ij}_t\big)_{t \geq 0}$,
$\big(\widetilde{B}^{ij}_t\big)_{t \geq 0}$,
$1 \leq i \leq N+\nu$, $1 \leq j \leq N$
be independent one-dimensional standard Brownian motions.
For $1 \leq i \leq j \leq N$, put
\begin{gather*}
S^{ij}_t= \begin{cases}
B^{ij}_t/\sqrt{2}, & i<j,
\cr
B^{ii}_t, & i=j,
\end{cases}
\qquad
A^{ij}_t= \begin{cases}
\widetilde{B}^{ij}_t/\sqrt{2}, & i<j,
\cr
0, & i=j,
\end{cases}
\end{gather*}
and let
$S^{ij}_t=S^{ji}_t$ (symmetric)
and $A^{ij}_t=-A^{ji}_t$ (anti-symmetric), $t \geq 0$
for $1 \leq j < i \leq N$.

\begin{Example}
 Put
 \begin{gather*}
 M_t= \big(M^{ij}_t\big) :=
 \big(S^{ij}_t+\sqrt{-1} A^{ij}_t\big)_{1 \leq i, j \leq N}, \qquad t \geq 0.
 \end{gather*}
 By definition
 $\langle {\rm d} M^{ij}, {\rm d} M^{k \ell} \rangle_t
 =\delta^{i \ell} \delta^{j k} {\rm d}t$, $t \geq 0$,
 $1 \leq i, j, k, \ell \leq N$.
 Hence, by unitarity of $U_t \in \sU_N$, $t \geq 0$,
 we see that $\Gamma^{ij, k \ell}_t = \delta^{i \ell} \delta^{j k}$,
 which gives
 $\langle {\rm d} \cM^i, {\rm d} \cM^j \rangle_{t}=\Gamma^{ii, jj}_t {\rm d}t =\delta^{ij} {\rm d}t$
 and $\Gamma^{ij, ji}_t \equiv 1$, $t \geq 0$, $1 \leq i, j \leq N$.
 Then Proposition \ref{thm:Bru} proves that
 the eigenvalue process $(\Lambda_t)_{t \geq 0}$,
 satisfies the following system of SDEs with $\beta=2$,
 \begin{gather}
 {\rm d} \Lambda^i_t
 = {\rm d} B^i_t + \frac{\beta}{2} \sum_{1 \leq j \leq N, j \not=i}
 \frac{{\rm d}t}{\Lambda^i_t-\Lambda^j_t},
 \qquad t \geq 0, \quad 1 \leq i \leq N.
 \label{eqn:Dyson1}
 \end{gather}
 Here $\big(B^i_t\big)_{t \geq 0}, 1 \leq i \leq N$ are
 independent one-dimensional standard Brownian motions,
 which are different from $(B^{ij}_t)_{t\geq 0}$
 and $\big(\widetilde{B}^{ij}_t\big)_{t \geq 0}$ used
 to define $\big(S^{ij}_t\big)_{t \geq 0}$ and $\big(A^{ij}_t\big)_{t \geq0}$,
 $1 \leq i, j \leq N$.
 If~$\beta=2$ and the initial configuration is $N \delta_0$,
 that is, all $N$ particles are at the origin, then at each time $t > 0$,
 $\Lambda_t =\big(\Lambda_t^1, \dots, \Lambda_t^N\big)$
 gives a point process on $\R$ which is equal in distribution
 with the GUE eigenvalue point process with variance
 $t$ \cite{Dys62,KT04}.
 For $\beta > 0$, we call the
 solution of~(\ref{eqn:Dyson1})
 the $N$-particle system of
 \textit{Dyson's Brownian motion model}
 with parameter $\beta$ \cite{Dys62},
 and write it as
 $\big(\Lambda^{\rD(N, \beta)}_t\big)_{t \geq 0}$.
\end{Example}

\begin{Example}
 Consider an $(N+\nu) \times N$ rectangular-matrix-valued process
 given by
 \begin{gather*}
 K_t :=\big(B^{ij}_t+\sqrt{-1} \widetilde{B}^{ij}_t\big)_{1 \leq i \leq N+\nu, 1 \leq j \leq N},
 \qquad t \geq 0,
 \end{gather*}
 and define an $\sH_N$-valued process by
 \begin{gather*}
 M_t=K^{\dagger}(t) K(t), \qquad t \geq 0.
 \end{gather*}
 The matrix $M_t$ is positive semi-definite and hence
 the eigenvalues are non-negative;
 $\Lambda^i_t \in \R_+$,
 $t \geq 0$, $1 \leq i \leq N$.
 We see that the finite-variation part of $dM^{ij}_t$ is
 equal to $2(N+\nu) \delta^{ij} {\rm d}t$, $t \geq 0$,
 and
 $\langle {\rm d} M^{ij}, {\rm d} M^{k \ell} \rangle_t=2\big(M^{i \ell}_t
 \delta^{jk}+M^{k j}_t \delta^{i \ell}\big) {\rm d}t$,
 $t \geq 0$, $1 \leq i, j, k, \ell \leq N$,
 which implies that
 ${\rm d} \Upsilon^i_t=2(N+\nu) {\rm d}t$,
 $\Gamma^{ij, ji}_t=2\big(\Lambda^i_t+\Lambda^j_t\big)$,
 and $\langle {\rm d} \cM^i, {\rm d} \cM^j \rangle_t = \Gamma^{ii, jj}_t {\rm d}t
 = 4 \Lambda^i_t \delta^{ij} {\rm d}t$,
 $t \geq 0$, $1 \leq i, j \leq N$.
 Then we have the following SDEs with $\beta=2$
 for the eigenvalue process of~$(M_t)_{t \geq 0}$,
 \begin{gather}
 {\rm d} \Lambda^i_t
 = 2 \sqrt{\Lambda^i_t} {\rm d} \widetilde{B}^i_t
 + \beta \Bigg[(\nu+1) + 2 \Lambda^i_t
 \sum_{1 \leq j \leq N, j \not=i}
 \frac{1}{\Lambda^i_t-\Lambda^j_t} \Bigg] {\rm d}t,
 \qquad t \geq 0, \quad 1 \leq i \leq N,
 \label{eqn:Wishart1}
 \end{gather}
 where $\big(\widetilde{B}^i_t\big)_{t \geq 0}, 1 \leq i \leq N$ are
 independent one-dimensional standard Brownian motions,
 which are different from $\big(B^{ij}_t\big)_{t \geq 0}$
 and $\big(\widetilde{B}^{ij}_t\big)_{t \geq 0}$,
 $1 \leq i, j \leq N$, used above
 to define the rectangular-matrix-valued process
 $(K_t)_{t \geq 0}$.
 The parameter $\nu$ can be extended to $\nu > -1$,
 in which if $\nu \in (-1, 0)$, a reflecting wall
 is put at the origin~\cite{KT04}.
 We call the solution of~(\ref{eqn:Wishart1}) the $N$-particle
 system of the \textit{Bru--Wishart process}
 with parameters $(\beta, \nu)$~\cite{Bru91},
 and write it as $\big(\Lambda_t^{\rBW(N, \beta, \nu)}\big)_{t \geq 0}$.

 The positive roots of eigenvalues of $M_t$ give the
 {\it singular values} of the rectangular matrix~$K_t$,
 which are denoted by
 \begin{gather}
 \cS^i_t :=\sqrt{\Lambda^i_t},
 \qquad t \geq 0, \quad 1 \leq i \leq N.
 \label{eqn:sqrt1}
 \end{gather}
 The system of SDEs for them is readily obtained
 from (\ref{eqn:Wishart1}) as
 \begin{gather}
{\rm d} \cS^i_t = {\rm d} \widetilde{B}^i_t
 + \frac{\beta(\nu+1)-1}{2 \cS^i_t} {\rm d}t
 + \frac{\beta}{2} \sum_{1 \leq j \leq N, j \not=i}
 \bigg( \frac{1}{\cS^i_t-\cS^j_t}
 + \frac{1}{\cS^i_t+\cS^j_t} \bigg) {\rm d}t,
 \label{eqn:Wishart2}
 \end{gather}
 $t \geq 0, 1 \leq i \leq N$
 with $\beta=2$ and $\nu > -1$.
 If $\beta=2$ and the initial configuration is $N \delta_0$,
 then at each time $t > 0$, $\cS_t=\big(\cS_t^1, \dots, \cS_t^N\big)$ on $\R_+$ gives
 the \textit{chiral GUE point process} with parameter~$\nu$
 and variance $t$
 studied in random matrix theory
 for high energy physics
 \cite{ABD11,For10,SV93,Ver94,VZ93}.
 For~$\beta > 0$, we call the solution of
 (\ref{eqn:Wishart2}) the \textit{chiral version of Dyson's
 Brownian motion model} with parameters $(\beta, \nu)$,
 and write it as $\big(\cS_t^{\rchD(N, \beta, \nu)}\big)_{t \geq 0}$.
\end{Example}


At each time $t >0$, the point processes
$\Lambda_t^{\rD(N, \beta)}$, $\Lambda_t^{\rBW(N, \beta, \nu)}$,
and $\cS_t^{\rchD(N, \beta, \nu)}$ are known as
typical examples of one-dimensional
\textit{log-gases} \cite{For10}.
Therefore, we will call the solutions of
the SDEs (\ref{eqn:Dyson1}), (\ref{eqn:Wishart1}),
and (\ref{eqn:Wishart2})
\textit{stochastic log-gases}.
Note that, when $\beta=2$, (\ref{eqn:Wishart1}) and
(\ref{eqn:Wishart2}) can be regarded as
the $N$-variable extensions of the
$2(\nu+1)$-dimensional squared Bessel process
and the Bessel process with parameter $\nu > -1$,
respectively \cite{Kat16}.

For
$\Lambda^{\rD(N, \beta)}_t
=\big(\Lambda^{\rD(N, \beta) \, 1}_t, \dots,
\Lambda^{\rD(N, \beta) \, N}_t\big)$, $t \geq 0$,
we regard the time evolution of empirical measures
\begin{gather*}
\Xi^{\rD(N, \beta)}_t(\cdot) :=\frac{1}{N} \sum_{i=1}^N
\delta_{\Lambda^{\rD(N, \beta) \, i}_t}(\cdot), \qquad t \in [0, T],
\end{gather*}
as an element of $\cC\big([0, T] \to \cP^0(\R)\big)$.
For
$\big(\Lambda^{\rBW(N, \beta, \nu)}_t\big)_{t \geq 0}$
and
$\big(\cS^{\rD(N, \beta, \nu)}_t\big)_{t \geq 0}$,
let
\begin{gather*}
\lambda := \frac{N+\nu}{N}\Longleftrightarrow\nu=(\lambda-1)N,
\end{gather*}
and consider
\begin{gather*}
\Xi^{\rBW(N, \beta, (\lambda-1)N)}_t(\cdot)
:=\frac{1}{N} \sum_{i=1}^N
\delta_{\Lambda^{\rBW(N, \beta, (\lambda-1)N) \, i}_t}(\cdot),
\qquad t \in [0, T]
\end{gather*}
as an element of $\cC\big([0, T] \to \cP^0(\R_+)\big)$, and
with (\ref{eqn:sqrt1}) consider
\begin{gather*}
\Sigma^{\rchD(N, \beta, (\lambda-1)N)}_t(\cdot)
:=\frac{1}{2N}\sum_{i=1}^N
\big\{\delta_{\cS^{\rchD(N, \beta, (\lambda-1)N) \, i}_t}(\cdot)
+\delta_{-\cS^{\rchD(N, \beta, (\lambda-1)N) \, i}_t}(\cdot) \big\},
\qquad t \in [0, T],
\end{gather*}
as an element of $\cC\big([0, T] \to \cP^0_{\rs}(\R)\big)$.
The following is proved~\cite{BNW13,CDG01,Cha92,RS93}.

\begin{Theorem}\label{thm:Burgers}
Assume that for any $N \in \N$,
the initial measures
$\Xi^{\rD(N, \beta)}_0$,
$\Xi^{\rBW(N, \beta, (\lambda-1)N)}_0$,
and
$\Sigma^{\rchD(N, \beta, (\lambda-1)N)}_0$ have
bounded supports, where
$\big(\Sigma^{\rchD(N, \beta, (\lambda-1)N)}_0\big)^{(2)}
=\Xi^{\rBW(N, \beta, (\lambda-1)N)}_0$
is satisfied,
and in $N \to \infty$ they converge weakly to the measures
$\rw_0 \in \cP^0(\R)$, $\rmm_{\lambda, 0} \in \cP^0(\R_+)$, and
$\rw_{\lambda, 0} \in \cP^0_{\rs}(\R)$, respectively.
Then for any fixed $T < \infty$,
\begin{gather*}
\big(\Xi^{\rD(N, \beta)}_t(\cdot)\big)_{t \in [0, T]}
\Longrightarrow (\rw_t(\cdot))_{t \in [0, T]}
\qquad \text{a.s. in}\quad \cC\big([0, T] \to \cP^0(\R)\big),
\\
\big(\Xi^{\rBW(N, \beta, (\lambda-1)N)}_t(\cdot)\big)_{t \in [0, T]}
\Longrightarrow (\rmm_{\lambda,t}(\cdot))_{t \in [0, T]}
\qquad \text{a.s. in}\quad \cC\big([0, T] \to \cP^0(\R_+)\big),
\\
\big(\Sigma^{\rchD(N, \beta, (\lambda-1)N)}_t(\cdot)\big)_{t \in [0, T]}
\Longrightarrow (\rw_{\lambda, t}(\cdot))_{t \in [0, T]}
\qquad \text{a.s. in}\quad \cC\big([0, T] \to \cP^0_{\rs}(\R)\big),
\end{gather*}
where $(\rw_t)_{t \geq 0}$,
$(\rmm_{\lambda, t})_{t \geq 0}$,
and $(\rw_{\lambda, t})_{t \geq 0}$ are
the time-dependent probability measures
defined so that their Cauchy transforms solve the PDEs
\eqref{eqn:Gwt}, \eqref{eqn:Gm}, and \eqref{eqn:Gm1/2}
under the initial probability measures
$\rw_0$, $\rmm_{\lambda, 0}$, and $\rw_{\lambda, 0}$,
respectively.
\end{Theorem}

Note that dependence on the parameter $\beta$
vanishes in the limit $N \to \infty$.

By the construction mentioned above,
the relation
\begin{gather*}
\big(\Sigma^{\rchD(N, \beta, (\lambda-1)N)}_t\big)^{(2)}
\dist= \Xi^{\rBW(N, \beta, (\lambda-1)N)}_t,
\qquad t \geq 0,
\end{gather*}
holds,
and then Theorem \ref{thm:Burgers} proves
the equality
$\rw_{\lambda, t}^{(2)} = \rmm_{\lambda, t}$, $t \geq 0$.
This is consistent with the definition
(\ref{eqn:def_rmm}).\pagebreak

\subsection{Expressions of second push-forward measures}

We will prove the following.
\begin{Lemma}
\label{thm:mu2}
For $\mu \in \cP^0_{\rs}(\R)$ and $\nu \in \cP^0(\R_+)$,
the following four statements are equivalent with each other,
\begin{gather*}
(i)\quad \mu^{(2)}=\nu,
\\
(ii)\quad G_{\mu}(z) = z G_{\nu}\big(z^2\big),
\\
(iii)\quad R_{\mu}(z)= R_{\nu} \bigg( \dfrac{z^2}{R_{\mu}(z)+1} \bigg),
\\
(iv)\quad S_{\mu}(z) = S_{\rd}(z) \sqrt{S_{\nu}(z)}.
\end{gather*}
\end{Lemma}

\begin{proof} Assume that $\mu({\rm d}x)$ (resp.~$\nu({\rm d}x)$) has
 a probability density function $\rho_{\mu}(x)$ (resp.~$\rho_{\nu}(x)$).
 Let $B \in \cB((0, \infty))$. Then
 \begin{gather*}
 \nu(B) = \int_{\R} 1_B(x) \rho_{\nu}(x)\, {\rm d}x
= \int_{\R_+} 1_B\big(x^2\big) \rho_{\nu}\big(x^2\big)\, {\rm d}x^2
 =\int_{\R} 1_B\big(x^2\big) \rho_{\nu}\big(x^2\big) |x| {\rm d}x.
 \end{gather*}
 On the other hand, by definition (\ref{eqn:mu_p}),
 $\mu^{(2)}(B)=\int_{\R} 1_B\big(x^2\big) \rho_{\mu}(x)\, {\rm d}x$.
 Hence
 \begin{gather}
(i) \iff \rho_{\mu}(x)=\rho_{\nu}\big(x^2\big) |x|, \qquad x \in \R.
 \label{eqn:rho1}
 \end{gather}
 Since (\ref{eqn:rho1}) implies
 the symmetry $\rho_{\mu}(-x)=\rho_{\mu}(x)$, $x \in \R$,
 we see that
 \begin{align*}
 G_{\mu}(z) &= \int_{\R} \frac{\rho_{\mu}(x)}{z-x} {\rm d}x
= \frac{1}{2} \bigg\{
 \int_{\R} \frac{\rho_{\mu}(-x)}{z-x} {\rm d}x
 +\int_{\R} \frac{\rho_{\mu}(x)}{z-x} {\rm d}x \bigg\}
 \\[1ex]
 &= \frac{1}{2} \bigg\{
 \int_{\R} \frac{\rho_{\mu}(x)}{z+x} {\rm d}x
 +\int_{\R} \frac{\rho_{\mu}(x)}{z-x} {\rm d}x \bigg\}
 =z \int_{\R} \frac{\rho_{\mu}(x)}{z^2-x^2} {\rm d}x.
 \end{align*}
 Then when (\ref{eqn:rho1}) is satisfied,
 \begin{gather*}
 G_{\mu}(z) = z \int_{\R} \frac{\rho_{\nu}\big(x^2\big) |x|}{z^2-x^2} {\rm d}x
 =z \int_{\R_+} \frac{\rho_{\nu}\big(x^2\big)}{z^2-x^2} {\rm d}x^2
 = z G_{\nu}\big(z^2\big) \iff (ii).
 \end{gather*}
 When $(ii)$ is satisfied,
 \begin{gather*}
 G_{\mu} \bigg( \frac{R_{\mu}(z)+1}{z} \bigg)
 =\frac{R_{\mu}(z)+1}{z}
 G_{\nu} \bigg( \bigg( \frac{R_{\mu}(z)+1}{z} \bigg)^2 \bigg)
 \end{gather*}
 holds. By (\ref{eqn:R1}), this implies
 \begin{align*}
z&= \frac{R_{\mu}(z)+1}{z}
 G_{\nu} \bigg( \bigg( \frac{R_{\mu}(z)+1}{z} \bigg)^2 \bigg)
 \iff G_{\nu} \bigg( \bigg( \frac{R_{\mu}(z)+1}{z} \bigg)^2 \bigg)
 = \frac{z^2}{R_{\mu}(z)+1}
 \\
& \iff \bigg( \frac{R_{\mu}(z)+1}{z} \bigg)^2
 =G_{\nu}^{\bra -1 \ket}
 \bigg( \frac{z^2}{R_{\mu}(z)+1} \bigg)
 = \frac{R_{\mu}(z)+1}{z^2}
 \bigg\{ R_{\nu} \bigg( \frac{z^2}{R_{\mu}(z)+1} \bigg)+1 \bigg\},
 \end{align*}
 where we used (\ref{eqn:R1}) again.
 This is equivalent with $(iii)$.

 When $(ii)$ is satisfied,
 \begin{gather}
 G_{\mu} \bigg( \frac{z+1}{z S_{\mu}(z)} \bigg)
 = \frac{z+1}{z S_{\mu}(z)}
 G_{\nu} \bigg( \bigg(\frac{z+1}{z S_{\mu}(z)} \bigg)^2 \bigg)
 \label{eqn:Eq1}
 \end{gather}
 holds.
 By (\ref{eqn:Psi1}), (\ref{eqn:chi_Psi1}), and (\ref{eqn:S_def1}),
 we can prove the equality
 \begin{gather*}
 G_{\mu} \bigg(\frac{z+1}{z S_{\mu}(z)} \bigg) = z S_{\mu}(z).
 \end{gather*}
 Then (\ref{eqn:Eq1}) gives
 \begin{gather}
 G_{\nu} \bigg(\bigg( \frac{z+1}{z S_{\mu}(z)} \bigg)^2 \bigg)
 = (z+1) \bigg( \frac{z S_{\mu}(z)}{z+1} \bigg)^2.
 \label{eqn:Eq2}
 \end{gather}
 By (\ref{eqn:Psi1}),
 the l.h.s.\ of (\ref{eqn:Eq2}) is equal to
 $\big[ \Psi_{\nu}\big(\{z S_{\mu}(z)/(z+1)\}^2\big)
 +1\big] \{z S_{\mu}(z)/(z+1)\}^2$.
 Hence we obtain the equalities
 $z= \Psi_{\nu}\big(\{z S_{\mu}(z)/(z+1)\}^2\big)$
 $\iff
 \chi_{\nu}(z)=\{z S_{\mu}(z)/(z+1) \}^2$.
 By (\ref{eqn:S_def1}), this gives
 \begin{gather*}
 S_{\mu}(z)^2=\frac{1+z}{z} S_{\nu}(z).
 \end{gather*}
 We use (\ref{eqn:Sda}) with $a=1$ and then
 $(iv)$ is obtained.
 Hence the proof is complete.
\end{proof}

\section{General solutions of complex Burgers-type equations}\label{sec:solution}

Let $t \in [0, \infty)$ and $z \in \C^+$ be
independent variables and consider a PDE
for a complex function $g=g(t,z) \in \C$ in the form,
\begin{gather}
A(t, z, g) \frac{\partial g}{\partial t}
+B(t, z, g) \frac{\partial g}{\partial z}
=C(t, z,g).
\label{eqn:PDE1}
\end{gather}
We regard the solution of (\ref{eqn:PDE1}) as a surface
$g=g(t, z)$ in the space $[0, \infty) \times \C^+ \times \C$.
Then (\ref{eqn:PDE1}) is interpreted as a
geometrical statement that the vector field
$(A(t,z,g), B(t, z, g), C(t, z, g))$ is tangent to the
surface at every point.
This statement means that the graph of solution is given by
a union of integral curves of this vector field.
They are called the \textit{characteristic curves} \cite{CH62} of (\ref{eqn:PDE1})
and satisfy the \textit{Lagrange--Charpit equatoin}
(see \cite{Del97} and references therein),
\begin{gather}
\frac{{\rm d}t}{A(t, z, g)}
=\frac{{\rm d}z}{B(t, z, g)}
=\frac{{\rm d}g}{C(t, z, g)}.
\label{eqn:LC1}
\end{gather}
Here we consider the special case such that
$A(t,z,g) \equiv 1$.
Then (\ref{eqn:LC1}) is written as
\begin{gather}
\begin{cases}
\displaystyle{\frac{{\rm d}z}{{\rm d}t} =B(t, z, g)},
\\[2ex]
\displaystyle{\frac{{\rm d}g}{{\rm d}t} =C(t, z, g)}.
\end{cases}
\label{eqn:eqs1}
\end{gather}
We will show that
the solutions of (\ref{eqn:eqs1}) for
(\ref{eqn:Gwt}), (\ref{eqn:Gm1/2}) and (\ref{eqn:Gm})
are obtained in the forms of functional equations.

\begin{Proposition}
\label{thm:GD}\quad
\begin{enumerate}\itemsep=0pt
\item[$(i)$]
Given the Cauchy transform $G_{\rw_0}(z)$ of
the initial measure $\rw_0 \in \cP^0(\R)$,
the solution of~\eqref{eqn:Gwt} satisfies the
functional equation,
\begin{gather}
G_{\rw_t}(z)=G_{\rw_0}(z-t G_{\rw_t}(z)),
\qquad t \geq 0.
\label{eqn:GD_func}
\end{gather}

\item[$(ii)$]
Given the Cauchy transform $G_{\rmm_{\lambda, 0}}(z)$
of the initial measure $\rmm_{\lambda, 0} \in \cP^0(\R_+)$,
the solution of \eqref{eqn:Gm} satisfies the
functional equation,
\begin{gather}
\frac{1}{G_{\rmm_{\lambda,t}}(z)}=
t + \frac{1}{G_{\rmm_{\lambda, 0}} (
(1-t G_{\rmm_{\lambda,t}}(z)) \{(1-\lambda)t
+(1-t G_{\rmm_{\lambda,t}}(z))z \}
)}, \quad t \geq 0.
\label{eqn:GBW_func}
\end{gather}

\item[$(iii)$]
Given the Cauchy transform
$G_{\rw_{\lambda, 0}}(z)$
of the initial measure $\rw_{\lambda, 0} \in \cP^0_{\rs}(\R)$,
the solution of~\eqref{eqn:Gm1/2} satisfies the
functional equation,
\begin{gather}
\frac{1}{G_{\rw_{\lambda,t}}(z)}
\!=\!\frac{t}{z}
\!+\! \frac{\sqrt{(1\!-\!\frac{t}{z} G_{\rw_{\lambda,t}}(z))
\{(1\!-\!\lambda)t \!+\!(1\!-\!\frac{t}{z} G_{\rw_{\lambda,t}}(z))z^2 \}}
}{zG_{\rw_{\lambda, 0}} \Big(
\sqrt{\big(1\!-\!\frac{t}{z} G_{\rw_{\lambda,t}}(z)\big)
\{(1\!-\!\lambda)t \!+\!\big(1\!-\!\frac{t}{z} G_{\rw_{\lambda,t}}(z)\big)z^2 \}}
\Big)}, \ \, t \geq 0.\!
\label{eqn:sGBW_func}
\end{gather}
\end{enumerate}
\end{Proposition}

\begin{proof}
 $(i)$ Consider the PDE (\ref{eqn:Gwt})
 for $g(t,z)=G_{\rw_t}(z)$.
 In this case (\ref{eqn:eqs1}) becomes
 \begin{gather}
 \frac{{\rm d}z(t)}{{\rm d}t} =g(t, z(t)),
 \label{eqn:GD_1}
 \\
\frac{{\rm d}g(t, z(t))}{{\rm d}t} =0.
 \label{eqn:GD_2}
 \end{gather}
 By (\ref{eqn:GD_2}), we can conclude that
 \begin{gather}
 g(t, z(t))=g(0, z(0)) \qquad \forall t \geq 0.
 \label{eqn:GD_3}
 \end{gather}
 Therefore, (\ref{eqn:GD_1}) is integrated as
 \begin{gather*}
 z(t) = z(0)+t g(0, z(0))
 = z(0)+ t g(t, z(t))
\Longleftrightarrow
z(0)=z(t)-t g(t, z(t)), \qquad t \geq 0.
 \end{gather*}
 Inserting this into (\ref{eqn:GD_3}), we obtain
 (\ref{eqn:GD_func}) for $g(t,z)=G_{\rw_t}(z)$.

 $(ii)$ Consider the PDE (\ref{eqn:Gm})
 for $g(t,z)=G_{\rmm_{\lambda, t}}(z)$.
 In this case (\ref{eqn:eqs1}) becomes
 \begin{align}
 &\frac{{\rm d}z(t)}{{\rm d}t} =2z g(t, z(t))-(1-\lambda),
 \label{eqn:GBW_1}
 \\
 &\frac{{\rm d}g(t, z(t))}{{\rm d}t} =-g(t, z(t))^2.
 \label{eqn:GBW_2}
 \end{align}
 The solution of (\ref{eqn:GBW_2}) is given by
 \begin{gather}
 g(t, z(t))=\frac{1}{t+1/g(0, z(0))}.
 \label{eqn:GBW_3}
 \end{gather}
 Then (\ref{eqn:GBW_1}) is written as
 \begin{gather*}
 \frac{{\rm d}z(t)}{{\rm d}t} = \frac{2 z(t)}{t+1/g(0, z(0))} - (1-\lambda).
 \end{gather*}
 This is integrated as
 \begin{gather}
 z(t)= \bigg( t + \frac{1}{g(0, z(0))} \bigg)
 \bigg\{ 1- \lambda
 + C \bigg( t + \frac{1}{g(0, z(0))} \bigg) \bigg\},
 \label{eqn:GBZ_4}
 \end{gather}
 where $C$ is an integral constant.
 By setting $t=0$ in this equation, we see that
 \begin{gather*}
 C=g(0, z(0)) \{ z(0) g(0, z(0)) - (1-\lambda) \}.
 \end{gather*}
 Using this and (\ref{eqn:GBW_3}),
 (\ref{eqn:GBZ_4}) is rewritten as
 \begin{gather*}
 z(t)=\frac{1}{g(t, z(t))}
 +\frac{\lambda t g(t, z(t))-1}{(1-t g(t, z(t))) g(t, z(t))}
 +\frac{z(0)}{(1-t g(t, z(t)))^2},
 \end{gather*}
 which gives
 \begin{gather*}
 z(0)=(1-t g(t, z(t)))
 \{(1-\lambda)t +(1-t g(t, z(t))) z(t) \}, \qquad t \geq 0.
 \end{gather*}
 If we insert this expression of $z(0)$ into (\ref{eqn:GBW_3})
 and replace $z(t)$ by $z$,
 $g(t, z(t))$ by $G_{\rmm_{\lambda, t}}(z)$,
 and~$g(0, \cdot)$ by $G_{\rmm_{\lambda, 0}}(\cdot)$, then
 we obtain (\ref{eqn:GBW_func}).

 $(iii)$ By Lemma \ref{thm:mu2}$(ii)$, (\ref{eqn:GBW_func})
 is transformed into (\ref{eqn:sGBW_func}).
\end{proof}

\section{Proofs}\label{sec:proof}

\subsection{Proof of Theorem \ref{thm:main}}\label{sec:proof_main}

\hspace*{5mm}$(i)$ We put $z=G^{\bra -1 \ket}_{\rw_t}(\zeta)$ in
(\ref{eqn:GD_func}). Then we have
$\zeta=G_{\rw_0}\big(G^{\bra -1 \ket}_{\rw_t}(\zeta)-t\zeta\big)$.
Next we apply~$G^{\bra -1 \ket}_{\rw_0}$ on both sides
and obtain
\begin{gather*}
G^{\bra -1 \ket}_{\rw_0}(\zeta)
=G^{\bra -1 \ket}_{\rw_t}(\zeta)-t\zeta
\Longleftrightarrow
\zeta G^{\bra -1 \ket}_{\rw_0}(\zeta) -1
=\big(\zeta G^{\bra -1 \ket}_{\rw_t}(\zeta)-1\big)-t\zeta^2.
\end{gather*}
By the definition (\ref{eqn:R1}),
this implies the following equation between $R$-transforms
\begin{gather*}
R_{\rw_0}(\zeta)=R_{\rw_t}(\zeta)-t \zeta^2, \qquad t \geq 0.
\end{gather*}
The assertion $(i)$ of Theorem \ref{thm:main} is concluded
by the well-known result \cite{HP00},
$R_{\rw^0_t}(z)= t z^2$.

$(ii)$
We put $z=G^{\bra -1 \ket}_{\rmm_{\lambda,t}}(\zeta)$ in
(\ref{eqn:GBW_func}). Then we have
\begin{align*}
\frac{1}{\zeta}
&= t + \frac{1}{
G_{\rmm_{\lambda,0}}\big((1-t \zeta) \big\{(1-\lambda) t +(1-t \zeta)
G^{\bra -1 \ket}_{\rmm_{\lambda,t}}(\zeta)\big\}\big)}
\nonumber\\
 &\Longleftrightarrow
G_{\rmm_{\lambda, 0}}\big((1-t \zeta) \big\{(1-\lambda) t +(1-t \zeta)
G^{\bra -1 \ket}_{\rmm_{\lambda,t}}(\zeta) \big\} \big)
=\frac{\zeta}{1-t \zeta}.
\end{align*}
We apply $G^{\bra -1 \ket}_{\rmm_{\lambda, 0}}$
on both sides and obtain
\begin{gather*}
(1-t \zeta) \big\{(1-\lambda) t +(1-t \zeta)
G^{\bra -1 \ket}_{\rmm_{\lambda,t}}(\zeta) \big\}
=G^{\bra -1 \ket}_{\rmm_{\lambda, 0}} \bigg( \frac{\zeta}{1-t \zeta} \bigg)
\\ \qquad
 {}\Longleftrightarrow
-\lambda \zeta t+(1-t \zeta) \big[
\zeta G^{\bra -1 \ket}_{\rmm_{\lambda, t}}(\zeta) -1 \big]
=\frac{\zeta}{1-t \zeta}
G^{\bra -1 \ket}_{\rmm_{\lambda, 0}} \bigg( \frac{\zeta}{1-t \zeta} \bigg)-1.
\end{gather*}
By the definition (\ref{eqn:R1}),
this implies the following equations between $R$-transforms,
\begin{align*}
- \lambda t z
+(1 - tz) R_{\rmm_{\lambda, t}}(z)
&=R_{\rmm_{\lambda, 0}} \bigg( \frac{z}{1-tz} \bigg)
\\
 &\Longleftrightarrow R_{\rmm_{\lambda, t}}(z)
=\frac{1}{1-t z} R_{\rmm_{\lambda, 0}} \bigg( \frac{z}{1-tz} \bigg)
+ \frac{\lambda t z}{1-t z}.
\end{align*}
Since
$R_{\rmm^0_{\lambda, t}}(z)= \lambda t z/(1 - t z)$
\cite{BNW13,HP00},
the assertion $(ii)$ is proved.

$(iii)$
Applying Lemma \ref{thm:mu2}$(iii)$ to (\ref{eqn:def_rmm}),
(\ref{eqn:W_Burgers_R_1}) gives
(\ref{eqn:chiral_Burgers_R_1}).

Hence the proof of Theorem \ref{thm:main} is complete.

\subsection{Proof of Theorem \ref{thm:main2}}

\hspace*{5mm}$(i)$ Since $\rmm_{\lambda, t} \in \cP^0(\R_+)$ is
defined by (\ref{eqn:def_rmm}), Lemma \ref{thm:mu2}$(iv)$
proves the assertion $(i)$.

$(ii)$
We start from (\ref{eqn:GD_func})
in Proposition \ref{thm:GD}$(i)$.
Replace $z$ by $1/z$ and then apply (\ref{eqn:Psi1}).
We~have
\begin{gather*}
z(\Psi_{\rw_t}(z)+1)
=
\frac{1}
{1/z-tz(\Psi_{\rw_t}(z)+1)}
\bigg\{ \Psi_{\rw_0} \bigg(
\frac{1}{1/z-t z(\Psi_{\rw_t}(z)+1)} \bigg)+1 \bigg\}.
\end{gather*}
Put $z=\Psi^{\bra -1 \ket}_{\rw_t}(\zeta)
=: \chi_{\rw_t}(\zeta)$. Then
\begin{gather*}
\chi_{\rw_t}(\zeta) (\zeta+1)
=
\frac{1}
{1/\chi_{\rw_t}(\zeta) - t \chi_{\rw_t}(\zeta) (\zeta+1)}
\bigg\{ \Psi_{\rw_0} \bigg(
\frac{1}{1/\chi_{\rw_t}(\zeta)
-t \chi_{\rw_t}(\zeta)(\zeta+1)} \bigg)+1 \bigg\}.
\end{gather*}
By (\ref{eqn:S_def1}),
the above is written as follows,
\begin{align*}
\zeta S_{\rw_t}(\zeta)&=
\frac{1}{\frac{\zeta+1}{\zeta S_{\rw_t}(\zeta)}
 - t \zeta S_{\rw_t}(\zeta)}
\bigg\{ \Psi_{\rw_0} \bigg(
\frac{1}{\frac{\zeta+1}{\zeta S_{\rw_t}(\zeta)}
-t \zeta S_{\rw_t}(\zeta)} \bigg)+1 \bigg\}
\\
&\Longleftrightarrow
\zeta \big\{ 1 - t \zeta S_{\rw_t}(\zeta)^2 \big\}
= \Psi_{\rw_0} \bigg(\frac{\zeta S_{\rw_t}(\zeta)}
{\zeta+1-t \zeta^2 S_{\rw_t}(\zeta)^2} \bigg).
\end{align*}
Now we apply $\chi_{\rw_0}$ on both sides
and obtain
\begin{gather*}
\chi_{\rw_0}\big(\zeta \big\{1- t \zeta S_{\rw_t}(\zeta)^2\big\}\big)
=\frac{\zeta S_{\rw_t}(\zeta)}{\zeta+1
-t \zeta^2 S_{\rw_t}(\zeta)^2}.
\end{gather*}
Again we use (\ref{eqn:S_def1}) and
replace the variable $\zeta$ by $z$.
Then we obtain
\begin{gather*}
\frac{S_{\rw_t}(z)}{1- t z S_{\rw_t}(z)^2}
=S_{\rw_0} \big(
z \big( 1- t z S_{\rw_t}(z)^2 \big) \big),
\qquad t \geq 0.
\end{gather*}
which is written as (\ref{eqn:S_time_Dyson})
by (\ref{eqn:S_wt}).

$(iii)$
We can prove (\ref{eqn:S_time_Wishart}) similarly to~(\ref{eqn:S_time_Dyson}) as shown above.

$(iv)$ By Lemma \ref{thm:mu2}$(iv)$, (\ref{eqn:S_time_Wishart})
is transformed to (\ref{eqn:S_time_sqrWishart}).

Hence the proof of Theorem \ref{thm:main2} is complete.

\subsection{Proof of Propisition \ref{thm:wa}}

It is easy to verify that
\begin{gather*}
R_{\rw_t^0}(z) =t z^2,
\\
R_{\rd_a}(z)=\frac{1}{2} \Big[ \sqrt{1+4 a^2 z^2} -1 \Big].
\end{gather*}
Then (\ref{eqn:R_wa}) is immediately concluded
from Theorem \ref{thm:main}$(i)$.
We have already obtained $S_{\rd_a}(z)$ as
(\ref{eqn:Sda}) and
$S_{\rw_t^0}(z)$ as (\ref{eqn:S_wt}).
Then (\ref{eqn:S_time_Dyson})
of Theorem \ref{thm:main2}$(ii)$ gives
\begin{gather*}
\frac{S_{\rw_t}(z)}{1-t z S_{\rw_t}(z)^2}
= \frac{1}{a}
\sqrt{\frac{z+1-tz^2 S_{\rw_t}(z)^2}
{z-t z^2 S_{\rw_t}(z)^2}}.
\end{gather*}
This is written as
\begin{gather*}
t^2 z^3 S_{\rw_t}(z)^4
-z \big\{2 t z + t + a^2\big\} S_{\rw_t}(z)^2+(z+1)=0,
\end{gather*}
which is solved by (\ref{eqn:S_wa}).

\subsection{Proof of Proposition \ref{thm:ma}}

It is easy to verify that
\begin{gather*}
R_{\rmm_{\lambda, t}^0}(z) =\frac{\lambda t z}{1-tz}
=\lambda \bigg( \frac{1}{1-tz} -1 \bigg),
\\
R_{\delta_b}(z) =b z.
\end{gather*}
Then (\ref{eqn:R_ma}) is immediately concluded
from Theorem \ref{thm:main}$(ii)$.
We can see that
$S_{\delta_b}(z)=1/b$.
Then (\ref{eqn:S_time_Wishart})
of Theorem \ref{thm:main2}$(iii)$ gives
\begin{gather*}
\frac{S_{\rmm_{\lambda, t}}(z)}
{(1-t z S_{\rmm_{\lambda, t}}(z))
\{ 1-t (\lambda + z) S_{\rmm_{\lambda, t}}(z) \}
}
=\frac{1}{b}.
\end{gather*}
This is written as
\begin{gather*}
t^2 z(z+\lambda) S_{\mu_t}(z)^2
-(2tz+t \lambda+b) S_{\mu_t}(z) +1=0.
\end{gather*}
which is solved by (\ref{eqn:S_ma}).

\section{Concluding remarks}\label{sec:remarks}

We list out some concluding remarks.

$1.$
In addition to the free Brownian motion \cite{Bia97}
and the free Wishart process \cite{CDM05}, Demni introduced
the \textit{free Jacobi process}
in \cite{Dem08}.
This process has two parameters
$\lambda$ and $\theta$.
He derived the following PDE for the
Cauchy transform of the measure-valued process
$(\rk_{\lambda, \theta, t})_{t \geq 0}$,
\begin{align}
\frac{\partial G_{\rk_{\lambda, \theta, t}}(z)}{\partial t}
&+[2 \lambda \theta z(1-z) G_{\rk_{\lambda, \theta, t}}(z)
+ \{ (2 \lambda \theta -1) z + \theta(1-\lambda) \}
] \frac{\partial G_{\rk_{\lambda, \theta, t}}(z)}{\partial z}
\nonumber\\
&+\{ \lambda \theta (1-2z) G_{\rk_{\lambda, \theta, t}}(z)
+(2 \lambda \theta-1) \} G_{\rk_{\lambda, \theta, t}}(z)=0.
\label{eqn:G_Jacobi}
\end{align}
Demni showed that this equation has
the \textit{stationary measure} given by
\begin{gather*}
\rk_{\lambda, \theta}({\rm d}x)
=\max \bigg(0, 1-\frac{1}{\lambda} \bigg) \delta_0({\rm d}x)
+\max \bigg( 0, 1-\frac{1-\theta}{\lambda \theta} \bigg) \delta_1({\rm d}x)
+g(x) 1_{[x_-, x_+]}(x)\, {\rm d}x,
\end{gather*}
where
\begin{gather*}
g(x)= \frac{\sqrt{(x-x_-)(x_+-x)}}{2 \lambda \theta \pi x (1-x) }
\qquad \text{with} \quad
x_{\pm}=\big(\sqrt{\theta(1-\lambda \theta)}
\pm \sqrt{\lambda \theta(1-\theta)} \big)^2.
\end{gather*}
The distribution with the density $g(x)$ is known
as the \textit{Kesten--McKay law} \cite{Kes59,McK81}.
Is it possible to solve the initial-value problem
for (\ref{eqn:G_Jacobi}) as we did in this paper?
When $\lambda=1$ and $\theta=1/2$,
(\ref{eqn:G_Jacobi}) is much simplified and given by
\begin{gather*}
\frac{\partial G_{\rk_{1, 1/2, t}}(z)}{\partial t}
+\frac{\partial}{\partial z}
\bigg\{ \frac{1}{2} z(1-z) G_{\rk_{1, 1/2, t}}(z)^2 \bigg\}=0.
\end{gather*}
This equation was solved by Demni, Hamdi, and Hmidi \cite{DHH12}
when the initial probability measure is given by $\delta_1$
and by Izumi and Ueda for general initial probability measure \cite{IU15}.
In these papers, the solutions are related with
the \textit{free unitary Brownian motion} \cite{Bia97}.
See also \cite{Dem17,Ham18} for further study.

Another PDE for measure-valued process
was reported in \cite{TT21}
\begin{align*}
\frac{\partial G_{\rt_{\alpha, c, t}}(z)}{\partial t}
&- \big\{ 2cz G_{\rt_{\alpha, c, t}}(z)+(z+2-\alpha) \big\}
\frac{\partial G_{\rt_{\alpha, c, t}}(z)}{\partial z}
-z \frac{\partial^2 G_{\rt_{\alpha, c, t}}(z)}{\partial z^2}
\nonumber\\
&- \big\{ c G_{\rt_{\alpha, c, t}}(z) +1 \big\} G_{\rt_{\alpha, c, t}}(z)=0,
\end{align*}
where $\alpha$ and $c$ are positive parameters.
This describes the hydrodynamic limit of the
Bru--Wishart (Laguerre) process
in a high temperature regime (see also \cite{NTT21}).
Notice that this equation involves a second-order derivative of
$G_{\rt_{\alpha, c, t}}(z)$ and hence it is regarded as
a \textit{viscous} Burgers-type equation.
As proved by \cite{Cha92,RS93}, the hydrodynamic limit
of Dyson's Brownian motion model is described by
the \textit{inviscid} Burgers equation (\ref{eqn:Gwt}).
In this paper we have studied only such inviscid cases
of Burgers-type equations
(see Theorem \ref{thm:Burgers} in Section \ref{sec:hydro}).
As shown by \cite{BBCL99,CL97,CL01}, however,
if we consider the system of SDEs which have the same
drift terms with Dyson's Brownian motion model
(\ref{eqn:Dyson1}) but the martingale terms are replaced as
${\rm d} B_t^i \to \sqrt{N} {\rm d}B_t^i$
(compare~\cite[equation~(5.93)]{CL97}
with~\cite[equation~(7)]{RS93}),
then we obtain the viscous Burgers equation in the hydrodynamic
limit (see~\cite[equation~(5.98)]{CL97}).
How can we solve such
viscous Burgers-type equations?

$2.$
Forrester and Grela \cite{FG16} studied
the hydrodynamic limits of the
\textit{circular ensemble} as well as
the Jacobi ensemble.
In the former case, they considered the following
type of Cauchy transform,
\[
{G^{\circ}_{\mu}(z)
=\frac{1}{2} \oint\cot \bigg( \frac{z-x}{2} \bigg)\mu({\rm d}x)}.
\]
This seems to be a \textit{trigonometric extension}
of the usual Cauchy transform (\ref{eqn:Cauchy1}),
since if we introduce a parameter $r>0$, then we see
$(1/2r) \cot ((z-x)/2r) \to 1/(z-x)$
as $r \to \infty$.
See also \cite{CL01,HS21}.
Some \textit{elliptic extensions}
of Cauchy-type transform
have been also considered in a recent study
of elliptic integrable systems \cite{BKL22,BLL21}.
Is it meaningful to consider
trigonometric and elliptic extensions of
free probability theory?

$3.$
In the present paper, we have studied
the complex Burgers-type equations
which are obtained in the hydrodynamic limits
of the stochastic log-gases studied in random
matrix theo\-ry~\cite{For10,Kat16}.
Recently, one of the present authors and Koshida
proposed a new construction of the
\textit{multiple
Schramm--Loewner evolutions} (SLEs)
driven by stochastic log-gases using the notion of
the \textit{coupling} between the multiple SLEs and
the \textit{Gaussian free fields} \cite{KK20,KK21a,KK21b}.
Then the infinite-slit limits of
the Loewner equations
studied by \cite{dMHS18,dMS16,HS21}
in the case that the slits are growing simultaneously
can be interpreted as the hydrodynamic limits
of the multiple SLEs~\cite{HK18} in the same context of
stochastic processes as explained in Section~\ref{sec:hydro}.
In the study of the hydrodynamic limits of the multiple SLEs,
exact solutions of the complex Burgers-type equations
for interesting initial conditions is very important~\cite{HK18}.
On the other hand, new connections between the
free probability theory and the Loewner chains
have been reported~\mbox{\cite{FHS20,HS21}}.
Moreover, an interesting discussion was given such that
the complex Burgers-type equations themselves can be
regarded as Loewner equations for certain
subordination processes~\cite{HS21}.
The method to solve the initial-value problems
for the complex Burgers-type equations
by deriving the functional equations
reported in this paper
shall be developed to analyze the hydrodynamic limits
of multiple SLEs.

\subsection*{Acknowledgements}

The present authors would like to thank
Shinji Koshida and Yoshimichi Ueda
for useful comments on the manuscript.
They are grateful to the anonymous referees for
valuable suggestions for future studies on this subject.
MK was supported by
the Grant-in-Aid for Scientific Research~(C) (No.~19K03674),
(B) (No.~18H01124),
(S) (No.~16H06338),
(A) (No.~21H04432)
of Japan Society for the Promotion of Science.
NS was supported by
the Grant-in-Aid for Scientific Research
(B) (No.~19H01791)
and
(C)(No.~19K03515)
of Japan Society for the Promotion of Science.

\pdfbookmark[1]{References}{ref}
\LastPageEnding

\end{document}